\documentclass[11pt,reqno]{amsart}

\usepackage[utf8]{inputenc}
\usepackage[margin=1.25in]{geometry}
\usepackage{hyperref}
\usepackage{appendix}
\usepackage{amsfonts}
\usepackage{amssymb}
\usepackage{stmaryrd} 
\usepackage{amsmath}
\usepackage{amsthm}
\usepackage{mathrsfs}

\theoremstyle{plain}
\newtheorem{theorem}{Theorem}[section]
\newtheorem{corollary}[theorem]{Corollary}
\newtheorem{lemma}[theorem]{Lemma}
\newtheorem{proposition}[theorem]{Proposition}
\theoremstyle{remark}

\newtheorem{Assumption}[theorem]{Assumption}
\newtheorem{Claim}[theorem]{Claim}
\newtheorem{remark}[theorem]{Remark}
\DeclareMathOperator{\sech}{sech}

\usepackage{biblatex} 
\numberwithin{equation}{section}
\title[Universal edge scaling limit ]{Universal edge scaling limit of discrete 1d random Schrödinger operator with vanishing potentials}

\author{Yi HAN}
\address{Department of Pure Mathematics and Mathematical Statistics, University of Cambridge.
}
\email{yh482@cam.ac.uk}
\thanks{Supported by EPSRC grant EP/W524141/1.}

\begin{document}

\maketitle

\begin{abstract}
Consider random Schrödinger operators $H_n$ defined on $[0,n]\cap\mathbb{Z}$ with zero boundary conditions:
$$
    (H_n\psi)_\ell=\psi_{\ell-1}+\psi_{\ell+1}+\sigma\frac{\mathfrak{a}(\ell)}{n^{\alpha}}\psi_{\ell},\quad \ell=1,\cdots,n,\quad \quad \psi_{0}=\psi_{n+1}=0,
$$
where $\sigma>0$ is a fixed constant, $\mathfrak{a}(\ell)$, $\ell=1,\cdots,n$, are i.i.d. random variables with mean $0$, variance $1$ and fast decay. The bulk scaling limit has been investigated in \cite{kritchevski2011scaling}: at the critical exponent $\alpha=
\frac{1}{2}$, the spectrum of $H_n$, centered at $E\in(-2,2)\setminus\{0\}$ and rescaled by $n$, converges to the $\operatorname{Sch}_\tau$ process and does not depend on the distribution of $\mathfrak{a}(\ell).$

We study the scaling limit at the edge. We show that at the critical value $\alpha=\frac{3}{2}$, if we center the spectrum at 2 and rescale by $n^2$, then the spectrum converges to a new random process depending on $\sigma$ but not the distribution of $\mathfrak{a}(\ell)$. We use two methods to describe this edge scaling limit. The first uses the method of moments, where we compute the Laplace transform of the point process, and represent it in terms of integrated local times of Brownian bridges. Then we show that the rescaled largest eigenvalues correspond to the lowest eigenvalues of the random Schrödinger operator $-\frac{d^2}{dx^2}+\sigma b_x'$ defined on $[0,1]$ with zero boundary condition, where $b_x$ is a standard Brownian motion. This allows us to compute precise left and right tails of the rescaled largest eigenvalue and compare them to Tracy-Widom beta laws.

We also show if we shift the potential $\mathfrak{a}(\ell)$ by a state-dependent constant and take $\alpha=\frac{1}{2}$, then for a particularly chosen state-dependent shift, the rescaled largest eigenvalues converge to the Tracy-Widom beta distribution.
\end{abstract}

\section{Introduction}

Consider the random Schrödinger operator $H_n$ defined on $(0,n]\cap\mathbb{Z}$,
\begin{equation}\label{randomschrodinger1.1}
    (H_n\psi)_\ell=\psi_{\ell-1}+\psi_{\ell+1}+\sigma\frac{\mathfrak{a}(\ell)}{n^{\alpha}}\psi_{\ell},\quad \ell=1,\cdots,n,\quad \quad \psi_{0}=\psi_{n+1}=0,
\end{equation}
where $\sigma>0$, and $\mathfrak{a}(\ell)$, $\ell=1,2,\cdots,n,\cdots$ are random variables that satisfy 

\begin{Assumption}\label{assumption1.1} The random variables $\mathfrak{a}(\ell)$, $\ell=1,2,\cdots,$ are independent and
\begin{enumerate}
    \item  $\mathbb{E}[\mathfrak{a}(\ell)]=0$ for each $\ell$, \item $\mathbb{E}[\mathfrak{a}(\ell)^2]$=1 for each $\ell$, and \item for some $C>0$ and $0<\gamma<2/3$, we have 
$$\mathbb{E}[|\mathfrak{a}(\ell)|^k]\leq C^k k^{\gamma k} ,\quad \text{ for all } \ell,k\in\mathbb{N}_+.$$
\end{enumerate}
\end{Assumption}

Denote by $\Lambda_n$ the (random) set of eigenvalues of $H_n$. Throughout this paper we assume the parameter $\alpha>0$ in the definition of \eqref{randomschrodinger1.1}, so the random potentials are vanishing in the $n\to\infty$ limit. In this scenario, it is well-known that the spectrum $\Lambda_n$ converges to $[-2,2]$ almost surely as $n\to\infty$, and the empirical measure supported on $\Lambda_n$ converges to the arcsine law
\begin{equation}\rho=\frac{1}{\sqrt{1-E^2/4}}1_{|E|<2}\end{equation}
on $[-2,2]$. In other words, the macroscopic statistics of $\Lambda_n$ converges to that of the free Laplacian in the large $n$ limit. Convergence of the integrated density of states towards the free Laplacian also holds for higher dimensional Anderson operators with vanishing potentials, see \cite{dolai2021ids}. It is thus of interest to zoom in at a finer scale and investigate second order fluctuations and obtain nontrivial scaling limits of $\Lambda_n$
around some energy $E\in [-2,2]$.

For energy in the bulk of the spectrum, that is for $E\in(-2,2)$, the scaling limit is much better understood. This is closely related to the phenomenon of Anderson localization \cite{anderson1958absence}, which, in the context of random Schrödinger operator on $\mathbb{Z}$ (potentials without decay),
\begin{equation}\label{randomschrodinger1.2}
    (H\psi)_\ell=\psi_{\ell-1}+\psi_{\ell+1}+\sigma\mathfrak{a}(\ell)\psi_{\ell},\quad \ell\in\mathbb{Z},\quad 
\end{equation}
predicts that the spectrum of $H$ is pure point and the eigenfunctions decay exponentially fast at infinity with probability one. Proof of such results have been obtained via different methods in the past forty years under various assumptions of the distribution of $\mathfrak{a}$, and also for some multidimensional Anderson models. The literature is extensive so we refer to the review \cite{kirsch2007invitation} for a comprehensive list of works in this direction. For random potentials with decaying variance, \cite{kiselev1998modified} considered the random Schrödinger operator
\begin{equation}\label{randomschrodinger1.3}
    (H\psi)_\ell=\psi_{\ell-1}+\psi_{\ell+1}+\sigma\frac{\mathfrak{a}(\ell)}{\ell^{\alpha}}\psi_{\ell},\quad \ell\in\mathbb{N}_+,\quad 
\end{equation}
and their result is that for $\alpha>\frac{1}{2},$ the spectrum of $H$ is absolutely continuous (hence no localization); for $\alpha\in(0,\frac{1}{2})$ the spectrum is pure point and eigenfunctions are square integrable; and for $\alpha=\frac{1}{2}$ a mixed spectrum may occur depending on the value of $\sigma$. 

The bulk scaling limit of the spectrum $\Lambda_n$ of the Anderson operator $H_n$ \eqref{randomschrodinger1.1}, centered at a bulk energy $E\subset(-2,2)$, also depends on the value of $\alpha$ in an essential way. When the potentials do not decay, i.e. we take $\alpha=0$ in \eqref{randomschrodinger1.1} and assuming $\mathfrak{a}(\ell)$'s have a bounded density with respect to Lebesgue measure, then $n(\Lambda_n-E)$ converges to an inhomogeneous Poisson process, for any $E$ such that the density of states function is not flat near $E$. See Minami \cite{minami1996local}, and also Germinet and Klopp \cite{germinet2014spectral} for more refined results on Poisson statistics. For potentials with fast decay, $\alpha>\frac{1}{2}$, the rescaled point process $n(\Lambda_n-E)$ converges to a deterministic point process called the clock process, independent of $\sigma$ and the random variables $\mathfrak{a}(\ell)$. For potentials with slow decay $\alpha\in(0,\frac{1}{2})$, Poisson statistics is expected for $n(\Lambda_n-E)$, yet the only rigorous proof in this regime up to now \cite{kotani2017poisson} concerns a continuous time model with vanishing coefficients, driven by Brownian motion, see also \cite{flores2023one}, Remark 2.5. Therefore we may safely conclude that, in the no decay and slow decay regime $\alpha\in[0,\frac{1}{2})$, a Poisson scaling limit is expected, yet the limit is not universal in that the proofs use the fact that the distribution of $\mathfrak{a}(\ell)$ is absolutely continuous with respect to the Lebesgue measure, or at least has no atoms, so the proof does not work for discrete probability laws like the Bernoulli distribution. In the regime $\alpha\in(\frac{1}{2},\infty)$ the limiting distribution is universal but less interesting, as the randomness is hidden from our scaling.

As the reader might expect, universality with respect to probability laws for bulk eigenvalue statistics only takes place at the critical value $\alpha=\frac{1}{2}$, which is the main result of \cite{kritchevski2011scaling}. For the decaying model \eqref{randomschrodinger1.1}, with the choice $z=E/2+i\sqrt{1-E^2/4}$ and $\rho=1=\sqrt{1-E^2/4}$, \cite{kritchevski2011scaling} proved that, for $\Lambda_n$ being the spectrum of the operator $H_n$, we have $\rho n(\Lambda_n-E)-\text{arg}(z^{2n+2})-\pi$ converges to the point process (with $\tau=(\sigma\rho)^2$)
\begin{equation}
    \operatorname{Sch}_\tau:=\{\lambda:\varphi^{\lambda/\tau}(\tau)\in 2\pi \mathbb{Z}\},
\end{equation} that is, $\operatorname{Sch}_\tau$ consists of the set of $\lambda$ such that $\varphi^{\lambda/\tau}(\tau)\in 2\pi\mathbb{Z}$
where $\varphi^\lambda$ is the solution to the SDE
\begin{equation}
    d\varphi^\lambda(t)=\lambda dt+d\mathcal{B}+\text{Re}[e^{-i\varphi^\lambda(t)}d\mathcal{W}],\quad\varphi^\lambda(0)=0, 
\end{equation} where $\mathcal{B}$ and $\mathcal{W}$ are real and complex standard Brownian motions. See also \cite{dumaz2023delocalized} for a different context in which $\operatorname{Sch}_\tau$ arises as the bulk scaling limit.
In \cite{kritchevski2011scaling} they also considered the random Schrödinger operator on $(0,n]\cap\mathbb{Z}$
\begin{equation}\label{randomschrodinger1.4}
    (H_n\psi)_\ell=\psi_{\ell-1}+\psi_{\ell+1}+\sigma\frac{\mathfrak{a}(\ell)}{\ell^{\frac{1}{2}}}\psi_{\ell},\quad \ell=1,\cdots,n,\quad \quad \psi_{0}=\psi_{n+1}=0,
\end{equation}
and they identified the point process scaling limit in the bulk as the sine beta process in random matrix theory \cite{valko2009continuum}. One may also consider random potentials with different vanishing profiles and get some other scaling limits \cite{han2023more}.

In this paper we are interested in the edge scaling limit of $H_n$ at $\pm 2$, which is the edge of the arcsine law. The main question we would like to ask is, for what value of $\alpha>0$, do we expect non-trivial edge scaling limits of $H_n$ at $\pm 2$ that is universal in terms of the distribution of $\mathfrak{a}(\ell)$? How do we rescale the spectrum of $H_n$ (that is, for what value of $c>0$ do we consider $n^c(\Lambda_n- 2)$), and how do we characterize the limiting random point process?

Before stating our main results, we mention two recent papers \cite{dolai2021ids} and \cite{kawaai2022limiting} where the authors considered $d$-dimensional random Schrödinger operators on $\mathbb{Z}^d$ with vanishing single site potentials that are multidimensional generalizations of \eqref{randomschrodinger1.1}, and they showed that when the density of the single site potentials has a particular expression (depending on $\alpha$), then the extremal eigenvalues of $H_n$ converge to an inhomogeneous Poisson process. However, the single site potentials in these papers are heavy tailed and does not cover any distribution that has a bounded support. In this paper we take the opposite direction by considering potentials with fast decay, and do not impose further restrictions on the distribution beyond Assumption \ref{assumption1.1}.

The main discovery of this paper is that edge scaling limit occurs at $\alpha=\frac{3}{2}$ and $c=2$. That is, for $\alpha=\frac{3}{2}$ in \eqref{randomschrodinger1.1}, the point process $n^2(\Lambda_n-2)$ converges to a nontrivial random point process on $\mathbb{R}$, and the point process depends only on $\sigma$ but not other properties of the random potential $\mathfrak{a}$. This point process appears to be new in the literature, and we use two different methods to characterize it: a Brownian local time interpretation via the method of moments in Section \ref{section1.2}, and a Schrödinger operator interpretation in Section \ref{section1.4}. The two methods yield complementary results, and the techniques are respectively inspired by \cite{gorin2018stochastic} and \cite{ramirez2011beta}, with quite nontrivial modifications made in the present paper. We also discuss the possibility of Tracy-Widom edge fluctuations for a modified model of \eqref{randomschrodinger1.1} in Section \ref{section1.3}.

There are very few existing literature on the fluctuations of top eigenvalues of random Schrödinger operators, see \cite{ojm/1300802705}, \cite{Cambronero1999TheGS}, \cite{mckean1994limit} and \cite{grenkova1983basic} for earlier works in continuous space, with particularly chosen random potentials. Our paper provides the first universality result on fluctuations of top eigenvalues, and the exponent $\frac{3}{2}$ has never been predicted before in the literature. Although we do not discuss the $\alpha<\frac{3}{2}$ case in this paper, we believe the edge scaling limit will be a Poisson distribution and leave the proof for a future work.

\subsection{Universal edge scaling limit}\label{section1.2}

As we assume $H_n$ has zero boundary conditions, the eigenvalues of $H_n$ are exactly the eigenvalues of the tridiagonal matrix
\begin{equation}\label{matrixreps}
  H_n=  \begin{pmatrix}
v_1&1&&&&\\
1&v_2&1&&&\\
&1&\ddots&\ddots&\\
&&\ddots&\ddots&1&\\
&&&1&v_{n-1}&1\\
&&&&1&v_n
    \end{pmatrix}
\end{equation}
where $v_k=\sigma \frac{\mathfrak{a}(k)}{n^\alpha}$, $k=1,\cdots,n$.
Then we can regard $H_n$ as a random matrix and use tools from random matrix theory to obtain the edge scaling limit. 
However, the analogy between random Schrödinger operator $H_n$ and classical random matrix ensembles such as the Gaussian $\beta$ ensemble holds in the bulk but not at the edge. In the bulk, typical eigenvalue spacing is $\frac{1}{n}$ for both the random Schrödinger operator \eqref{randomschrodinger1.1} and an $n\times n$ GUE matrix normalized by $2\sqrt{n}$ ; but at the edge, an explicit computation of eigenvalues of $H_n$ (in the deterministic case $\sigma=0$) predicts that eigenvalue spacing has order $n^{-2}$ whereas the Tracy-Widom law of GUE matrix shows the eigenvalue spacing is $n^{-\frac{2}{3}}$ at the edge.

Although a direct connection to random matrix ensembles does not hold in our setting, some powerful techniques in random matrix theory can still be used. In this section we use the moment method to compute the eigenvalue distribution of $H_n$ at scale $n^{-2}$. The strategy is that we compute the $n^2$-th power of $H_n/2$ and try to find a scaling limit. We aim for an $\alpha>0$ such that the scaling limit is nontrivial and universal in terms of the distribution of $\mathfrak{a}(\ell)$, and our computation shows that for $\alpha=\frac{3}{2}$ this is precisely the case. This value is in contrast to the critical value $\alpha=\frac{1}{2}$ for a universal scaling limit in the bulk. The moment method has been used to establish edge scaling limits for a number of random matrix ensembles including \cite{feldheim2010universality},\cite{sodin2010spectral}, and to derive an alternative description of the stochastic Airy operator in \cite{gorin2018stochastic}. In our proof we will follow some computations in Gorin and Shkolnikov \cite{gorin2018stochastic} in the case of Gaussian beta ensembles (thanks to Dumitriu and Edelman \cite{dumitriu2002matrix}, the Gaussian beta ensemble admits a tridiagonal matrix representation, which is a starting point in \cite{gorin2018stochastic} and \cite{ramirez2011beta}), yet as explained above, our scaling limit is very different from the Gaussian ensembles.

Now we state the main results. Consider a probability space supporting a Brownian motion $W$, and for any $T>0$ define a stochastic kernel $K(x,y;T)$ on $[0,1]\times[0,1]$ via 
\begin{equation}
    K(x,y;T)=\frac{1}{\sqrt{2\pi T}}\exp(-\frac{(x-y)^2}{2T})\mathbb{E}_{B^{x,y}}\left[1_{\forall t:B^{x,y(t)}\in[0,1]}\exp\left(\frac{\sigma}{2}\int_0^1 L_a(B^{x,y})dW(a)\right)
    \right]
\end{equation}
where $B^{x,y}$ denotes a standard Brownian bridge that starts from $x$ at $t=0$ ad ends at $y$ at $t=T$, and independent of $W$; $L_a(B^{x,y})$ is the local time of $B^{x,y}$ at level $a$ on time interval $[0,T]$; with expectation $\mathbb{E}_{B^{x,y}}$ taking over $B^{x,y}$. 
Denote by $\mathcal{U}(T),T>0$ as the integral operator on $L^2([0,1])$ with respect to the kernel $K(x,y;T)$. Then the following properties of $\mathcal{U}(T)$ will be proved in Section \ref{section4.1}. After writing the first draft, we learned about the recent paper \cite{10.1214/21-EJP654}, where similar properties for Schrödinger operators with Gaussian potentials are derived in a more general setting. Thus we do not claim any originality for this result but keep the proof for sake of completeness.

\begin{proposition}\label{proposition1.2}
\begin{enumerate}
\item    Given any $T>0$, $\mathcal{U}(T)$ is symmetric, non-negative trace class operator on $L^2([0,1])$ with probability one, that satisfies
   \begin{equation}\label{traceups}
       \operatorname{Trace}(\mathcal{U}(T))=\int_0^1 K(x,x;T)dx.
   \end{equation}

\item    The operators $\mathcal{U}(T),T\geq 0$ has the semigroup property with probability one: for given $T_1,T_2\geq 0$, then $\mathcal{U}(T_1)\mathcal{U}(T_2)=\mathcal{U}(T_1+T_2)$ almost surely.

\item     The semigroup $\mathcal{U}(T),T\geq 0$, is strongly continuous in $L^2$ in the sense that given $T\geq 0$, $f\in L^2([0,1]),$ we have 
    $\lim_{t\to T}\mathbb{E}[\|\mathcal{U}(T)f-\mathcal{U}(t)f\|^2]=0$.

\item We can find a (random) orthogonal basis of vectors $\mathbf{v}^1,\mathbf{v}^2,\cdots\in L^2([0,1])$, as well as stochastic variables $\eta^1\geq \eta^2\geq\cdots$ on the same probability space in such a way that for any $T>0$, the spectrum of $\mathcal{U}(T)$ is given by $\exp(T\eta^i/2),i\in\mathbb{N}$ that correspond to eigenvectors $\mathbf{v}^i,i\in\mathbb{N}$.  
\end{enumerate}
\end{proposition}

Let $H_n$ denote the matrix representation \eqref{matrixreps} of the random Schrödinger operator \eqref{randomschrodinger1.1} and consider the $n\times n$ matrix
\begin{equation}
    \mathcal{M}(T,n)=\frac{1}{2}\left(\left(\frac{H_n}{2}\right)^{\lfloor Tn^2\rfloor}+\left(\frac{H_n}{2}\right)^{\lfloor Tn^2\rfloor -1}\right),
\end{equation} then we have the main convergence theorem

\begin{theorem}\label{theorem1.3ups}
    Under Assumption \eqref{assumption1.1} and assuming $\alpha=\frac{3}{2}$, we have 
    $$\lim_{n\to\infty}\mathcal{M}(T,n)=\mathcal{U}(T),\quad T\geq 0$$ in the sense that we have

\begin{enumerate}
    \item Convergence in weak sense: for any $f,g\in L^2([0,1])$ and $T>0$, denote by $\pi_nf$ the vector in $\mathbb{R}^n$ with components $n^{\frac{1}{2}}\int_{\frac{i-1}{n}}^{\frac{i}{n}}f(x)dx,$ $i=1,\cdots,n$, and $(\pi_nf)'$ the transpose of $\pi_nf$, then
    $$\lim_{n\to\infty}(\pi_nf)'\mathcal{M}(T,n)(\pi_ng)=\int_0^1 (\mathcal{U}(T)f)(x)g(x)dx
    $$ with the convergence in distribution and in the sense of moments.
    \item Convergence for traces: given any $T>0$,
    $$\lim_{n\to\infty}\operatorname{Trace}(\mathcal{M}(T,n))=\operatorname{Trace}(\mathcal{U}(T))$$ in distribution and in moments.
    \item The above convergence holds simultaneously for any finitely many $T$'s, $f$'s
 and $g$'s.
 \end{enumerate}
Moreover, the Brownian motion $W$ in the definition of $\mathcal{U}(T)$ can be realized as the following limit in the Skorokhod topology:
\begin{equation}\label{208}W(a):=\lim_{n\to\infty}\sum_{l=0}^{\lfloor na\rfloor} \frac{\mathfrak{a}(l)}{n^{1/2}},\quad  a\in[0,1].
\end{equation}
\end{theorem}

The trace convergence of $\mathcal{M}(T,n)$ leads to the Laplace transform of rescaled largest eigenvalues of $H_{n}$.

\begin{corollary}\label{corollary1.4} Take $\alpha=\frac{3}{2}$ and assume Assumption \ref{assumption1.1} holds. Denote by $\eta^1_n\geq\eta^2_n\geq\eta^n_n $ the eigenvalues of $n^{2}(H_n-2)$, where $H_n$ is the matrix representation \eqref{matrixreps} of the random Schrödinger operator \ref{randomschrodinger1.1}.
Then we have the following convergence result, in terms of Laplace transform, for the edge scaling limit of the random Schrödinger operator $H_n$:  
\begin{equation}\label{laplace}
    \sum_{i=1}^n
e^{T\eta_{n}^i/2}\to_{n\to\infty}\sum_{i=1}^\infty e^{T\eta^i/2}=\operatorname{Trace}(\mathcal{U}(T)),
\end{equation}
with the convergence holding jointly for finitely many $T$'s. In particular, we have convergence
\begin{equation}
    \eta^i_{n}\to_{n\to\infty }\eta^i
\end{equation}
simultaneously for finitely many $i$'s.

\end{corollary}

Since \eqref{laplace} holds for every $T>0$, this gives the convergence of Laplace transform for rescaled eigenvalues of $H_n$ towards $\mathcal{U}(T)$. This characterization involves all the eigenvalues of $H_n$, while a characterization to be given in Section \ref{section1.4} only involves the finitely many largest eigenvalues. Another interesting feature of this convergence is that it is pathwise, or deterministic in the sense that the Brownian motion $W$ in the definition of  $\mathcal{U}(T)$ can be realized as a deterministic function of the $\mathfrak{a}(\cdot)$'s via \eqref{208}. On the contrary, the interpretation to be given in Section \ref{section1.4} is more probabilistic, yielding sharp quantitative estimates but does not appear to have a pathwise interpretation.

In the degenerate case $\sigma=0$, where no randomness appears and $H_n$ is a deterministic matrix, we deduce the following interesting corollary:
\begin{corollary}
    For each $T>0$, the following limit holds
    \begin{equation}\label{ellipticfunction}
        \sum_{j=1}^\infty e^{-\frac{T}{2}\pi^2j^2}= \frac{1}{\sqrt{2\pi T}}\int_0^1 \mathbb{P}_{B^{x,x}}\left[\text{ for any } t\in[0,T],B^{x,x}(t)\in[0,1]\right] dx,
    \end{equation}
    where $B^{x,x}$ is a Brownian bridge connecting $x$ to $x$ on $[0,T]$.
\end{corollary}

\begin{proof} Set $\sigma=0$.
    Since $H_n$ is a Toeplitz matrix, its eigenvalues are $2\cos(\frac{k\pi}{n+1})$ for $k=1,2,\cdots,n$. Then for each $j$, $\lim_{n\to\infty}\eta_n^j=-j^2\pi^2$. Then the claim follows from Corollary \ref{corollary1.4} and dominated convergence theorem.
\end{proof}
The summation in \eqref{ellipticfunction} is the Jacobi theta function in the theory of elliptic functions, and it does not have a closed form unless $T$ takes some particular values. See also \cite{biane2001probability} or other references for Brownian motion representations of infinite series summations.

\subsection{Tracy-Widom fluctuations for shifted means}\label{section1.3}

An interesting related problem is if we can construct a 1d random Schrödinger operator of the form \eqref{randomschrodinger1.1} that has similar behavior  as a Gaussian random matrix at the edge, or to say, to obtain Tracy-Widom fluctuations for \eqref{randomschrodinger1.1}.

Unfortunately, this does not seem possible from the author's eyes, if we insist that the random variables $\mathfrak{a}(\ell)$ are i.i.d. with zero mean and unit variance. An example of 2d random Schrödinger operator whose lowest eigenvalue has Tracy-Widom fluctuation was constructed in \cite{kotowski2019tracy} via establishing a connection between the lowest eigenvalues of the random Schrödinger operator and the fluctuation of the free energy of a log-gamma polymer \cite{krishnan2018tracy}. The polymer model is integrable, whose free energy has GUE Tracy-Widom fluctuation, and hence the Schrödinger operator has Tracy-Widom fluctuation at the edge.

In this note we observe that, if we add a deterministic, state-dependent shift to the variable $\mathfrak{a}(\ell)$'s, then we can obtain Tracy-Widom ($\beta$)-fluctuations for a series of random Schrödinger operators for any $\beta>0$. In this case we take $\alpha=\frac{1}{2}$ in \eqref{randomschrodinger1.1} and we rescale eigenvalues by $n^{2/3}$. This scaling is consistent with the one that appears in Gaussian random matrix ensembles.

Recall that for any $\beta>0$, the Tracy-Widom $(\beta)$-distribution is defined in the celebrated work \cite{ramirez2011beta}, as the negative of the lowest eigenvalue of the following random Schrödinger operator on $[0,\infty)$ with zero boundary condition at $0$:
\begin{equation}\label{starbeta}
    \mathcal{H}_\beta=-\frac{d^2}{dx^2}+x+\frac{2}{\sqrt{\beta}}b',
\end{equation}
where $b'$ is a white noise. The irregularity of $b'$ requires us to interpret $\mathcal{H}_\beta$ via integration by parts, and the details can be found in \cite{ramirez2011beta}. This definition generalizes the definition of Tracy-Widom laws in the integrable case $\beta=1,2,4$ in \cite{tracy1994level}, \cite{tracy1996orthogonal}.

\begin{proposition}\label{theorem 1.6}
For each $\beta>0$, consider the random Schrödinger operator
\begin{equation}\label{randomschrodingerT.1}
({H^{\beta}_n}\psi)_\ell=\psi_{\ell-1}+\psi_{\ell+1}+\left(\frac{2}{\sqrt{\beta}}\frac{\mathfrak{a}(\ell)}{n^{1/2}}-\frac{\ell}{n}\right)\psi_{\ell},\quad \ell=1,\cdots,n,\quad \quad \psi_{0}=\psi_{n+1}=0,
\end{equation}
where $\mathfrak{a}(1),\cdots,\mathfrak{a}(n),\cdots$ are i.i.d. mean $0$, variance $1$ random variables and have bounded fourth moments. Denote by $\lambda_{1,n}^\beta>\lambda_{2,n}^\beta>\cdots$ the eigenvalues of $H_n^\beta$, then for each $k\in\mathbb{N}_+$, we have the joint convergence in distribution of the random vector
$$\left(n^{2/3}\left(2-\lambda_{1,n}^\beta\right),\cdots,n^{2/3}\left(2-\lambda_{k,n}^\beta\right)\right)$$
in the limit $n\to\infty$ towards the lowest $k$ eigenvalues of the operator $\mathcal{H}_\beta$ defined in \eqref{starbeta}. 

More generally, for any sequence $m_n$ of positive integers tending to infinity with $m_n=o(n)$, consider 
\begin{equation}\label{randomschrodingerT.12}
({H^{\beta}_n}\psi)_\ell=\psi_{\ell-1}+\psi_{\ell+1}+\left(\frac{2}{\sqrt{\beta}}\frac{\mathfrak{a}(\ell)}{(m_n)^{3/2}}-\frac{\ell}{(m_n)^3}\right)\psi_{\ell},\quad \ell=1,\cdots,n,\quad \quad \psi_{0}=\psi_{n+1}=0,
\end{equation}
then the random vector 
$$\left((m_n)^2\left(2-\lambda_{1,n}^\beta\right),\cdots,(m_n)^2\left(2-\lambda_{k,n}^\beta\right)\right)$$
converges in the $n\to\infty$ limit to the lowest $k$ eigenvalues of $\mathcal{H}_\beta$.

\end{proposition}
The method to drive Proposition \ref{theorem 1.6} is outlined in Section \ref{section5}.

\subsection{A Schrödinger operator representation of the edge scaling limit}\label{section1.4}

In this section we turn back to the random Schrödinger operator \eqref{randomschrodinger1.1} with potentials of zero mean, variance 1 satisfying Assumption \ref{assumption1.1}. We have determined, for $\alpha=\frac{3}{2}$, the (universal) edge scaling limit of $H_n$ via the Laplace transform (see \eqref{laplace} and \eqref{traceups}). It is interesting to ask if we can find a random Schrödinger operator representation for this scaling limit, just as the operator $\mathcal{H}_\beta$ in \eqref{starbeta} governs the Tracy-Widom $\beta$-distribution. We give a positive answer by considering the following random operator defined on $[0,1]$ with zero boundary conditions:
\begin{equation}
    \mathcal{G}_\sigma=-\frac{d^2}{dx^2}+\sigma b_x'
\end{equation}

More precisely, let $L^*$ be the set of functions $f:[0,1]\to\mathbb{R}$ with $f(0)=0,f(1)=0$ and $\int_0^1 (|f|^2+|f'|^2)dx<\infty.$  We say that $(\psi,\lambda)\in L^*\times\mathbb{R}$ is a pair of eigenfunction/eigenvalue of $\mathcal{G}_\sigma$ if $\|\psi\|_2=1$ and 
\begin{equation}
    \psi''(x)=\sigma \psi(x)b_x'-\lambda\psi(x),
\end{equation}
in the sense that the following integration by parts formula holds
\begin{equation}
    \psi'(x)-\psi'(0)=\sigma\psi(x)b_x-\int_0^x \sigma b_y\psi'(y)dy-\int_0^x \lambda\psi(y)dy.
\end{equation}

Since the random operator $\mathcal{G}_\sigma$ is defined on $[0,1]$, its eigenvalues are almost surely bounded from below, and we know that, by \cite{Fukushima1977OnSO}, Section 2 or \cite{doi:10.1142/S0219199715500820}:
\begin{proposition}\label{proposition1.9}
    Almost surely, for any $k\geq 0$, the set of eigenvalues of $\mathcal{G}_\sigma$ has a well-defined $(k+1)$-st smallest eigenvalue $\Lambda_k$. The eigenvalues of $\mathcal{G}_\sigma$ are distinct admitting no accumulation point. Moreover, $\Lambda_k\to\infty$ as $k\to\infty$ with probability 1.
\end{proposition}

This proposition yields the following variational characterization of the minimal eigenvalue of $\mathcal{G}_\sigma$:
\begin{equation}\label{variational}
\Lambda_0=\inf_{f\in L^*,\|f\|_2=1}\left\{\sigma \int_0^1 f^2(x)db(x)+\int_0^1 [f'(x)]^2dx\right\}
\end{equation} where $b$ is a standard Brownian motion and the stochastic integral is defined through integration by parts.

Now we state the main result of this section, which claims that the rescaled largest eigenvalues of the random Schrödinger operator $H_n$ \eqref{randomschrodinger1.1} converges to the eigenvalues of $\mathcal{G}_\alpha$:

\begin{theorem}\label{jamstheorem}
  Assume $\alpha=\frac{3}{2}$ and Assumption \ref{assumption1.1} is satisfied.  Let $\lambda_1^n\geq\lambda_2^n\geq\cdots$ denote the eigenvalues of the random Schrödinger operator $H_n$, \eqref{randomschrodinger1.1}. Then for any $k\in\mathbb{N}_+$ the vector 
  \begin{equation}
      \left(n^2(2-\lambda_\ell^n)\right)_{\ell=1,\cdots,k}
  \end{equation} converges in distribution in the $n\to\infty$ limit to the random vector $(\Lambda_0,\Lambda_1,\cdots,\Lambda_{k-1})$.
\end{theorem}

Indeed, as the proof in Section \ref{sec6.3nomoment} will show, the third requirement in Assumption \ref{assumption1.1} is not necessary for the proof of Theorem \ref{jamstheorem}; instead a uniform bound on the fourth moment of $\mathfrak{a}(\ell)$ is sufficient. A possible explanation of this discrepancy is that Theorem  \ref{jamstheorem} only concerns the convergence of a finite number of eigenvalues, whereas the moment method in Section \ref{section1.2} yields Laplace transform \eqref{laplace} that depends on all the eigenvalues.

Our proof technique of a random matrix converging to a random Schrödinger operator also has some similarity to the recent works \cite{10.1214/20-AIHP1055} and \cite{PACHECO2017114}. Further localization properties of the limiting operator $\mathcal{G}_\sigma$ have been investigated in \cite{WOS:000519778200008} and a series of works that follow.

 Combined with results from Section \ref{section1.2}, we can set up the following equivalence between the random operators $\mathcal{U}(T), T>0$ with the random operator $\mathcal{G}_\sigma$. The correspondence is detailed in the following corollary, which will be proved in Section \ref{section6.4}. We learned after writing the first draft that \cite{10.1214/21-EJP654}, Theorem 2.24 also contains some more general result on equivalence of operators of this form.

\begin{corollary}\label{corollary1.11}
For any $T>0$ define $e^{-\frac{T}{2}\mathcal{G}_\sigma}$ the unique operator on $L^2([0,1])$ with respect to the orthogonal basis spanned by eigenvectors of $\mathcal{G}_\sigma$ and having corresponding eigenvalues $e^{-T\Lambda_0/2}\geq e^{-T\Lambda_1/2}\geq\cdots$. We take a coupling of $e^{-\frac{T}{2}\mathcal{G}_\sigma}$ with $\mathcal{U}(T)$ via identifying the Brownian motions $W$ that appear in their definitions. Then given any $T>0$, the operator $e^{-\frac{T}{2}\mathcal{G}_\sigma}$ and $\mathcal{U}(T)$ coincide almost surely.
\end{corollary}

It is also very interesting to investigate the left and right tails of the smallest eigenvalue of the random operator $\mathcal{G}_\sigma$ via the variational characterization \eqref{variational}. For the Tracy-Widom distribution, a rather detailed asymptotic analysis can be found in \cite{ramirez2011beta}, Section 4. In our case, we have the following result, which will be proved in Section \ref{section7}:

\begin{theorem}\label{theorem1.11}
Let $\operatorname{RSO}_\sigma:=-\Lambda_0(\sigma)$ where $\Lambda_0(\sigma)$ is the smallest eigenvalue of $\mathcal{G}_\sigma$. Then for $a\uparrow \infty$ we have
\begin{equation}\begin{aligned}
&\mathbb{P}\left(\operatorname{RSO}_\sigma>a\right)=\exp\left(-\frac{8}{3\sigma^2}a^{3/2}(1+o(1))\right),\quad \text{and}\\
&\mathbb{P}\left(\operatorname{RSO}_\sigma<-a\right)=\exp\left(-\frac{1}{2\sigma^2}a^2(1+o(1))\right).
\end{aligned}\end{equation}
\end{theorem}
This is to be compared with the Tracy-Widom $(\beta)$-distribution, whose tails have asymptotic (\cite{ramirez2011beta}, Theorem 1.3): as $a\uparrow\infty,$
$$\begin{aligned}
&\mathbb{P}\left(\operatorname{TW}_\beta>a\right)=\exp\left(-\frac{2}{3}\beta a^{3/2}(1+o(1))\right),\quad\text{ and }\\
&\mathbb{P}\left(\operatorname{TW}_\beta<-a\right)=\exp\left(-\frac{1}{24}\beta a^3(1+o(1))\right).
\end{aligned}$$

Some parts of Theorem \ref{theorem1.11} have appeared in previous literature. For the right tail, a related estimate was first derived in \cite{Fukushima1977OnSO}, and a more precise probability estimate in the right tail was claimed in \cite{https://doi.org/10.1002/cpa.20104} that implies our right tail estimate via a simple asymptotic computation. Yet our computation is of an independent interest as it is thus much shorter and simpler than that in \cite{https://doi.org/10.1002/cpa.20104} or any other possible references. It can also be generalized to other settings without much difficulty.

We observe that (setting $\sigma=\frac{2}{\sqrt{\beta}}$), $\operatorname{RSO}_\sigma$ has the same right tail decay as $\operatorname{TW}_\beta$ up to the precision stated in the theorem, yet the left tail of $\operatorname{RSO}_\sigma$ is Gaussian-like and decays slower than the left tail of $\operatorname{TW}_\beta$. It would be interesting to find a better understanding of the mechanism that accounts for this similarity and difference between the tails of $\operatorname{RSO}_\sigma$ and  $\operatorname{TW}_\beta$.

Finally, we mention that all the proofs of edge scaling limit of random Schrödinger operators in this paper work only for the discrete model defined on $\mathbb{Z}$, \eqref{randomschrodinger1.1}, where a random matrix interpretation \eqref{matrixreps} is available. It is very natural to conjecture that for 1d Schrödinger operator in continuous space 
\begin{equation}\label{continuoustime}
    H_{\alpha,n}=-\frac{d^2}{dt^2}+n^{-\alpha}F(X),\quad \text{on } L^2([0,n]),
\end{equation} where $F$ is a smooth function on a manifold $M$ and $X$ is the Brownian motion on $M$,
we have the same scaling limit, or at least a nontrivial scaling limit occurs at $\alpha=\frac{3}{2}$. We also expect that similar edge scaling limits should occur for higher dimensional Anderson operators defined on the lattice with vanishing potentials, such that the limit does not depend on the potential beyond its first few moments (in contrast to \cite{dolai2021ids}, \cite{kawaai2022limiting}). A rigorous proof of these results needs many additional efforts beyond the techniques presented in this paper, so we leave these questions for a future investigation. Yet we note that the case \eqref{continuoustime} can possibly be treated via a method similar to the universality result in \cite{krishnapur2016universality}.

\section{Proof sketch}
We first use the method of moments to show why a universal scaling limit is expected at $\alpha=\frac{3}{2}$. This section serves as a sketch of proof of results announced in Section \ref{section1.2}, while the details are presented in Section \ref{section3}. The ideas presented here are largely inspired by \cite{gorin2018stochastic}.

In the rest, the symbol $H_n$ stands for the tridiagonal matrix \eqref{matrixreps}, which is the matrix representation of the random Schrödinger operator \eqref{randomschrodinger1.1}.
We will compute very high powers of the matrix $H_n$, indeed compute its power up to $n^2$. By definition of matrix product,
\begin{equation}
    (H_n)^k[i,i']=\sum H_n[i_0,i_1]H_n[i_1,i_2]\cdots H_n[i_{k-2},i_{k-1}]H_n[i_{k-1},i_k],
\end{equation}
with summation over integer sequences $i_0,i_1,\cdots,i_k$ in $\{1,2,\cdots,n\}$ with $i_0=i$, $i_k=i'$ and $|i_j-i_{j-1}|\leq 1$ for all $i,j=1,2,\cdots,k$.

As an illustration we first consider $\left(\frac{H_n}{2}\right)^k$ and $k=\lfloor Tn^2\rfloor$. Assume $k$ is even and $|i-i'|$ is even, then the corresponding part of the sum \textbf{involving no diagonal element} is
\begin{equation}\frac{1}{2^k}
    \sum_{\substack{1\leq i_0,i_1,\cdots,i_k\leq n\\|i_j-i_{j-1}|=1\forall j\\i_0=i,i_k=i'}}1.
\end{equation}
Note that by the constraint of $k$ and $|i-i|'$, the number of diagonal elements in the sum must be even. The sum involving \textbf{precisely two diagonal elements} is given by 
$$\frac{1}{2^k}\sum_{\substack{1\leq i_0,i_1,\cdots,i_{k-2}\leq n\\|i_j-i_{j-1}|=1\forall j\\i_0=i,i_{k-2}=i'}}1^{k-2}\times \left(\frac{1}{n^{2\alpha}}\sum_{0\leq j\leq \ell\leq k-2}\mathfrak{a}(i_j)\mathfrak{a}(i_l)\right). 
$$
We rewrite the last factor in terms of $\frac{1}{n^{2\alpha}}(\sum_{j=0}^{k-2}\mathfrak{a}(i_j))^2$ and $\frac{1}{n^{2\alpha}}\sum_{j=0}^{k-2}\mathfrak{a}(i_j)^2$, and we denote them by $A$ and $B$ respectively. Recall that $k\sim Tn^2$. Assuming $\alpha$ is very large, say $\alpha>1$, then $B$ vanishes in expectation as $n\to\infty$ so it only suffices to consider $A$ in the scaling limit. We may rewrite
\begin{equation}\label{summationindex}
\sum_{j=0}^{k-2}\mathfrak{a}(i_j)=\sum_{l=0}^n \mathfrak{a}(l)\cdot \left|\{j=1,\cdots,k:i_j=l\}\right|.
\end{equation}
A trajectory of simple random walk connecting $i_0$ to $i_{k-2}$ with $k$ steps typically visit $k^{1/2}\sim n$ order of sites, and the occupation time of each site typically has order $k^{1/2}\sim n$. Therefore we may regard the summation $\frac{1}{n^\alpha}\sum_{j=0}^{k-2}$ as a summation of $n$ independent, mean 0, variance 1 random variables with coefficient $\frac{1}{n^{\alpha-1}}$ each. By the central limit theorem, the sum converges to a Gaussian only if $\alpha=\frac{3}{2}$.

For technical convenience, it would be useful to consider simultaneously the sum involving any odd number of diagonal elements. To cover that case, we consider $(\frac{H_n}{2})^{\lfloor Tn^2-1\rfloor}$ and do the same expansion. To take both even and odd cases into account, we introduce as in the introduction  $\mathcal{M}(T,n):=\frac{1}{2}\left(\left(\frac{H_n}{2}\right)^{\lfloor Tn^2\rfloor}+\left(\frac{H_n}{2}\right)^{\lfloor Tn^2-1\rfloor}\right)$ and prove convergence of $\mathcal{M}(T,n)$ as $n$ tends to infinity.

The rest of the proof of results announced in Section \ref{section1.2} follows from computing summations involving any even and odd number of diagonal elements, and approximating the result as an exponential. 

For the proof of results in Section \ref{section1.3} and \ref{section1.4}, we adapt the procedure in \cite{ramirez2011beta}, yet we have to rewrite most parts of the proof because our limiting random Schrödinger operator $\mathcal{G}_\sigma$ (and its domain) is very different from the stochastic Airy operator $\mathcal{H}_\beta$ considered in \cite{ramirez2011beta}.

\section{Proof via the method of moments}\label{section3}

\subsection{Approximating random walk by Brownian bridges}
Given $x,y\in\mathbb{R}$ and $n\in\mathbb{N}$, $\widetilde{T}>0$ so that $\widetilde{T}n^2$ is an integer with parity of $\lfloor nx\rfloor-\lfloor ny\rfloor.$ Write $$X^{x,y;n,\widetilde{T}}:=(X^{x,y;n,\widetilde{T}}(0),X^{x,y;n,\widetilde{T}}( n^{-2}),\cdots,X^{x,y;n,\widetilde{T}}(\widetilde{T}))$$ as the simple random walk bridge that links $\lfloor nx\rfloor$ to $\lfloor ny\rfloor$ in $\widetilde{T}n^2$ steps of increment size $\pm 1$, such that the walk does not go below $0$ or go above $n$ within time $[0,\tilde{T}]$. The trajectory of $X^{x,y;n,\widetilde{T}}$ is obtained uniformly over all trajectories with these properties. Define also the occupation times, up to normalization
\begin{equation}
    L_h(X^{x,y;n,T})=n^{-1}\left|\{t\in[0,\widetilde{T}]:X^{x,y;n,\widetilde{T}}(t)=nh\}\right|,\quad h\in\mathbb{R}.
\end{equation}

The following coupling result is crucial throughout the rest of the proof.

\begin{proposition}\label{proposition3.1} Given $x,y\in\mathbb{R}$, consider $T_n,n\in\mathbb{N}$ a sequence of positive numbers with $\sup_n |T_n-T|n^2<\infty$ for a given $T>0$, and that $T_nn^2\in\mathbb{N}$. Then we can find a probability space that supports a sequence of random walk bridges $X^{x,y;n,T_n},$ $n\in\mathbb{N}$, a standard Brownian bridge $B^{x,y}$ on $[0,T]$ that connects $x$ to $y$, and a random variable $\mathcal{C}$ such that for any $n\in\mathbb{N}$,
\begin{equation}\label{3.03.2}
    \sup_{h\in\mathbb{R}}|L_h(X^{x,y;n,T_n})-L_h(B^{x,y})|\leq Cn^{-3/16}
\end{equation}
\begin{equation}\label{3.13.3}
    \sup_{0\leq t\leq T_n\wedge T}|n^{-1}X^{x,y;n,T_n}(t)-B^{x,y}(t)|\leq \mathcal{C}n^{-1}\log n.
\end{equation}
\end{proposition}

The proof is deferred to Appendix \ref{AppendixA}, and is inspired by \cite{gorin2018stochastic}, Proposition 4.1.

We will also need the following auxiliary convergence result, stating that while the random walk bridge converges, the summation involving local times also converge:

\begin{proposition}
Given $x,y\in\mathbb{R}$ and $T_n>0$ with $\sup_n |T_n-T|n^2<\infty$ for a given $T>0$, that $T n^2\in\mathbb{N}$ for all $n$, and $Tn^2$ having same parity with $\lfloor nx\rfloor-\lfloor ny\rfloor.$ Then we have the joint convergence of 
\begin{equation}
    \left( n^{-1}X^{x,y;n,T_n},\sigma\sum_{h\in n^{-1}(\mathbb{Z}_{\geq 0}+\frac{1}{2})} L_h(X^{x,y;n,T_n}) \frac{\mathfrak{a}(\lfloor nh\rfloor)}{n^{1/2}}\right)
\end{equation}
as $n\to\infty$, converges in distribution to \begin{equation}
    \left(B^{x,y},\sigma\int_0^\infty L_a(B^{x,y})dW_{\mathfrak{a}}(a)\right)
\end{equation}

\end{proposition}

\begin{proof} Consider the conditional characteristic function \begin{equation}
    \mathbb{E}_{\mathfrak{a}}\left[\exp\left(i\sigma u\sum_{h\in n^{-1}(\mathbb{Z}_{\geq 0}+\frac{1}{2})}L_h(X^{x,y;n,T_n})\frac{\mathfrak{a}(\lfloor nh\rfloor)}{n^{1/2}}\right)\right],\quad u\in\mathbb{R}
\end{equation}
with the expectation taken over $\mathfrak{a}(m),m\in\mathbb{N}$. Define also the conditional characteristic function \begin{equation}
    \mathbb{E}_{W_\mathfrak{a}}\left[\exp\left(i\sigma u\int_0^\infty L_a(B^{x,y})dW_{\mathfrak{a}}(a)\right)\right],\quad u\in\mathbb{R}
\end{equation}
we show the difference between them converge to $0$ as $n\to\infty$, under the coupling discussed in Proposition \ref{proposition3.1}. For fixed $B^{x,y}$, the variable $\sigma \int_0^\infty L_a(B^{x,y})dW_{\mathfrak{a}}(a)$ is a mean zero normal with variance $\sigma^2 \int_0^\infty L_a(B^{x,y})^2da$; meanwhile for fixed $X^{x,y;n,T_n}$, 
$$\sum_{h\in n^{-1}(\mathbb{Z}_\geq 0+\frac{1}{2})} L_h(X^{x,y;n,T_n})\frac{\mathfrak{a}(\lfloor nh\rfloor)}{n^{1/2}}$$ 
is a sum of independent variables, so the desired convergence will follow from central limit theorem once we verify, in the limit $n\to\infty$,
\begin{equation}\label{first}
\sum_{h\in n^{-1}(\mathbb{Z}_{\geq 0}+\frac{1}{2})}\mathbb{E}_{\mathfrak{a}}\left[L_h(X^{x,y;n,T_n})\frac{\mathfrak{a}(\lfloor nh\rfloor)}{n^{1/2}}\right]\to 0,
\end{equation}

\begin{equation}\label{second}
    \sum_{h\in n^{-1}(\mathbb{Z}_{\geq 0}+\frac{1}{2})}\mathbb{E}_{\mathfrak{a}}\left[\left(L_h\left(X^{x,y;n,T_n}\right)\sigma\frac{\mathfrak{a}(\lfloor nh\rfloor)}{n^{1/2}}\right)^2\right]\to \sigma^2\int_0^\infty \left[L_a(B^{x,y})\right]^2da,
\end{equation}

\begin{equation}\label{third}
    \sum_{h\in n^{-1}(\mathbb{Z}_{\geq 0}+\frac{1}{2})}\mathbb{E}_{\mathfrak{a}}\left[\left| L_h(X^{x,y;n,T_n}) \sigma\frac{\mathfrak{a}(\lfloor nh\rfloor)}{n^{1/2}}\right|^3\right]\to 0.
\end{equation} By the coupling in Proposition \ref{proposition3.1}, the rescaled local times $L_h(x^{x,y;n,T_n})$ are bounded in a way that is uniform for $h\in\mathbb{R}$ and is zero for $h$ sufficiently large. Thus
\eqref{first} and \eqref{third} follow from the fact that $\mathfrak{a}(\cdot)$ has mean zero, the fact that $L_h(X^{x,y;n,T_n})$ and $L_h(B^{x,y})$ are bounded uniformly in $n$ for all given $h$ thanks to the coupling in Proposition \ref{proposition3.1}, and that there are only $n$ terms in the summation of $|\mathfrak{a}(\cdot)|^3.$ To check \eqref{second}, compute
$$\begin{aligned}\Bigg|&\sum_{h\in n^{-1}(\mathbb{Z}_{\geq 0}+\frac{1}{2})}\mathbb{E}_{\mathfrak{a}}\left[\left(L_h (X^{x,y;n,T_n})\frac{\mathfrak{a}(\lfloor nh\rfloor)}{n^{1/2}}\right)^2\right]\\&-\sum_{h\in n^{-1}(\mathbb{Z}_{\geq 0}+\frac{1}{2})}\frac{L_h(B^{x,y})^2\mathbb{E}[\mathfrak{a}(\lfloor nh\rfloor)^2]}{n}\Bigg|\\\leq& 2\max_{h\in n^{-1}(\mathbb{Z}_{\geq 0}+\frac{1}{2})}\max\{(L_h(X^{x,y;n,T_n}),L_h(B^{x,y})\}\mathcal{C}n^{-3/16}\\&\cdot \frac{1}{n}\sum_{ n^{-1}(\mathbb{Z}_{\geq 0}+\frac{1}{2})\ni h\leq\max(\|X^{x,y;n,T_n}\|_\infty,\|B^{x,y}\|_\infty)}\mathbb{E}[\mathfrak{a}(\lfloor nh\rfloor)^2].
\end{aligned}
$$
The expression almost surely converges to $0$ thanks to the fact that $h\to L_h(X^{x,y;n,T_n})$, $h\to L_h(B^{x,y})$, and $\|X^{x,y;n,T_n}\|_\infty$ are bounded uniformly and there are $n$ $\mathfrak{a}(\cdot)$ to be summed. Finally, since $a\to L_a(B^{x,y})^2$ is almost surely uniformly continuous in $a$, we have the almost sure convergence
$$\sum_{h\in n^{-1}(\mathbb{Z}_{\geq 0}+\frac{1}{2})}\sigma^2\frac{\left[L_h\left(B^{x,y}\right)]^2\mathbb{E}[\zeta\left(\lfloor nh\rfloor\right)^2\right]}{n}\to \sigma^2\int_0^\infty [L_a(B^{x,y})]^2da.
$$
which implies \eqref{second} and completes the proof of the whole proposition.

\end{proof}

We also need the following version with random initial condition and multiple Brownian bridges. First introduce the following notation: given probability distributions $\lambda,\mu$ on $[0,1]$ and $\tilde{T}\in n^{-2}\mathbb{N}$, we denote by $X^{\lambda,\mu;n,\widetilde{T}}$ a random walk bridge whose start and end points $x,y$ are independently sampled from the image of $\lambda$, $\mu$ via the map $z\to \lfloor nz\rfloor$, and that the walk makes $\widetilde{T}n^2$ steps of $\pm 1$ if $\widetilde{T}n^2$ has the same parity as $\lfloor nx\rfloor-\lfloor ny\rfloor$, and $\widetilde{T}n^2-1$ steps if this is not the case. Write also $B^{\lambda,\mu}$ the Brownian bridge whose start and end points are independently sampled from $\lambda$ and $\mu$. We prove the following generalization:

\begin{proposition}\label{generalization}
    For a given $R\in\mathbb{N}$, given any probability distributions $\lambda_1,\cdots,\lambda_R$, $\mu_1,\cdots,\mu_R$ supported on $[0,1)$, a sequence of positive reals $T_n,n\in\mathbb{N}$ with $\sup_n |T_n-T|n^2<\infty$ for a given $T$, and $T_nn^2\in\mathbb{N},N\in\mathbb{N}$.  Given $n\in\mathbb{N}$, consider $X^{\lambda_1,\mu_1;n,T_n},\cdots,X^{\lambda_R,\mu_R;n,T_n}$ that are $R$ independent bridges of random walk. Then the $R$-dimensional random vector
    \begin{equation}
           \left(\frac{1}{n}X^{\lambda_r,\mu_r;n,T_n},
            \sigma \sum_{n^{-1}(\mathbb{Z}+\frac{1}{2})\ni h\leq 1}L_h(X^{\lambda_r,\mu_r;n,T_n)}\frac{\mathfrak{a}(\lfloor nh\rfloor)}{n^{1/2}}\right)_{r=1}^R,
    \end{equation}converges as $n\to\infty$ in law to 
    \begin{equation}
        \left(B^{\lambda_r,\mu_r},\sigma \int_0^\infty L_a(B^{\lambda_r,\mu_r})dW_{\mathfrak{a}}(a)\right)_{r=1}^R,
    \end{equation}
for $B^{\lambda_1,\mu_1},\cdots, B^{\lambda_R,\mu_R}$ independent Brownian bridges on $[0,T]$, and $W_{\mathfrak{a}}$ an independent standard Brownian motion. We also have the convergence in Skorokhod topology 
\begin{equation}\label{4.38}
    \frac{1}{ n^{1/2}}\sum_{m=1}^{\lfloor an\rfloor} \mathfrak{a}(m)\to W_{\mathfrak{a}}(a),\quad n\to\infty,\quad a\in[0,1].
\end{equation}
\end{proposition}

\begin{proof} In the case $R=1,$ replacing fixed endpoints by random distributions follow from integrating over $\lambda_1$, $\mu_1$. For $R>1$, one uses the same proof by applying the multidimensional version of central limit theorem. By the same CLT we obtain the convergence \eqref{4.38} at the level of finite dimensional marginals. Then we enhance this, via tightness, to convergence on the process level for the Skorokhod topology.
\end{proof}

We note that when using the method of moments, we will consider random walks from $\lfloor nx\rfloor$ to $\lfloor ny\rfloor$ within $\widetilde{T}n^2$ steps which never leaves $[0,n]$. The latter structural constraint is not very easy to eliminate via combinatorial identities, so we temporarily forget about this constraint and instead enumerate over all possible random walk paths, as in the following lemma. We will take care of this constraint in the limit by conditioning on the event that the Brownian bridge $B^{x,y}$ never leaves $[0,1]$.

We deduce the limits of path counting numbers as follows: Consider 
$$\Xi(x,y;n,\widetilde{T}):=\frac{n}{2^{\widetilde{T}n^2}}\begin{pmatrix}\widetilde{T}n^2\\\frac{1}{2}(\widetilde{T}n^2+\lfloor nx\rfloor-\lfloor ny\rfloor)\end{pmatrix},$$ which gives the number of all trajectories $X^{x,y;n,\widetilde{T}}$, re-scaled such that it has magnitude $O(1)$. Then

\begin{lemma}\label{combinatorialbound}
    Given $T_n$, $n\in\mathbb{N}$ with $\sup_n |T_n-T|n^2<\infty$, assume also $T_nn^2\in\mathbb{N},n\in\mathbb{N}$. Then the following convergence holds uniformly for $x,y$ in $[0,1]$ with $\lfloor nx\rfloor-\lfloor ny\rfloor$ having the same parity as $Tn^2$:
    \begin{equation}
        \Xi(x,y;n,T_n)\mapsto \sqrt{\frac{2}{\pi T}}e^{-(x-y)^2/(2T)},\quad n\to\infty.
    \end{equation}
    Moreover we have the (finite $n$) estimate
    \begin{equation}
        \Xi(x,y;n,\widetilde{T})\leq Ce^{-(x-y)^2/(C\widetilde{T})}
    \end{equation} given $x,y\in\mathbb{R},n\in\mathbb{N}$, any $\widetilde{T}>0$ and $\widetilde{T}n^2\in\mathbb{N}$ having the same parity with $\lfloor nx\rfloor-\lfloor ny\rfloor.$
\end{lemma}
The first claim follows from the de Moivre-Laplace Theorem and the second claim needs a bit more reasoning, and the proof is exactly the same as in \cite{gorin2018stochastic}, Lemma 4.2.

\subsection{Convergence of leading term contributions}

Define the trace process
$$\begin{aligned}
\operatorname{Tr}(n)&=\frac{1}{2}\int_0^1 \sum_{j=0}^{\lfloor Tn^2\rfloor}\left(\Xi(x,x;n,n^{-2}(\lfloor Tn^2\rfloor-j-\epsilon_{j,x,x}))\rfloor\right)\\&\cdot\mathbb{E}_X \left[1_{\forall t:0\leq X(t)\leq n}\frac{1}{(2n^\alpha)^j} \sum_{0\leq i_1\leq\cdots\leq i_j\leq \lfloor Tn^2\rfloor-j-\epsilon_{j,x,x}}\prod_{j'=1}^j \mathfrak{a}(X(i_{j'}n^{-2}))\right]dx,
\end{aligned}
$$
whence $\epsilon_{j,x,y}\in\{0,1\}$ is selected such that $\lfloor Tn^2\rfloor-j-\epsilon_{j,x,y}$ is always an even integer, and $X$ denotes the trajectory $X^{x,y;n,n^{-2}(\lfloor Tn^2\rfloor-j-\epsilon_{j,x,y})}$. 

Define also the scalar product
$$\begin{aligned} \operatorname{Sc}(n)=&\frac{1}{2}\int_0^1\int_0^1 f(x)g(y)\\&\cdot\sum_{j=0}^{\lfloor Tn^2\rfloor}\Xi\left(x,y;n,(\lfloor Tn^2\rfloor-j-\epsilon_{j,x,y})n^{-2}\right)\mathbb{E}_X\Bigg[1_{\forall t:0\leq X(t)\leq n}\\&\cdot\frac{1}{(2n^{\alpha)^j}}\sum_{0\leq i_1\leq \cdots\leq i_j\leq \lfloor Tn^2\rfloor-j-\epsilon_{j,x,y}}\prod_{j'=1}^j \mathfrak{a}(X(i_{j'}n^{-2}))\Bigg]dxdy.
\end{aligned}
$$
Then we see that  $\operatorname{Sc}(n)=(\pi_nf)'\mathcal{M}(T,n)(\pi_ng)$ and $\operatorname{Tr}(n)=\operatorname{Trace}(\mathcal{M}(T,n))$, recalling their definitions stated in Theorem \ref{theorem1.3ups}. Therefore, to prove Theorem \ref{theorem1.3ups}, it suffices to prove the $n\to\infty$ convergence of $\operatorname{Tr}(n)$ and $\operatorname{Sc}(n)$.

For technical convenience, define also the truncated $j$-th order contribution via
\begin{equation}\label{564}\begin{aligned} \operatorname{Sc}^{(j)}(n;\underline{R},\overline{R})=&\frac{1}{2}\int_0^1\int_0^1 f(x)g(y)\\&\cdot\sum_{j=0}^{\lfloor Tn^2\rfloor}\Xi\left(x,y;n,T_n\right)\mathbb{E}_X\Bigg[1_{\forall t:0\leq X(t)\leq n}\\&\cdot\underline{R}\vee\frac{(\sum_{i'=0}^{T_nn^2}\left(\mathfrak{a}(X(i'n^{-2})))\right)^j}{j!(2n^{\alpha })^j}\wedge\overline{R}\Bigg] dxdy.
\end{aligned},
\end{equation}
where in the definition of $\operatorname{Sc}^{(j)}(n;\underline{R},\overline{{R}}),$ the time $\{T_n\}$ satisfies $\sup_n |T_n-T|n^2<\infty$ for some $T>0$. moreover, $T_nn^2$ is an integer for each $n$. In the proof of the next proposition we shall take $T_n=(\lfloor Tn^2\rfloor-j-\epsilon_{j,x,y})n^{-2}$, but the precise value of $T_n$ is actually not important as they shall lead to the same limit.

The aim of this subsection is to prove the following Proposition, as we are interested in the case $\underline{R}=-\infty$ and $\overline{R}=\infty$:
\begin{proposition}\label{proposition3.5} Assume $\alpha=\frac{3}{2}$ and Assumption \ref{assumption1.1} holds. For a finite collection of $j$'s, some $\underline{R}\in[-\infty,0]$ and $\overline{R}\in[0,\infty]$, we have in the limit $n\to\infty$ the following convergence in distribution and in moments
$$\begin{aligned}
\operatorname{Sc}^{(j)}(n;\underline{R},\overline{R})\to &\frac{1}{\sqrt{2\pi T}}\int_0^1\int_0^1 f(x)g(y)\exp(-\frac{(x-y)^2}{2T})\\&\cdot \mathbb{E}_{B^{x,y}}[1_{\forall t:B^{x,y}(t)\in [0,1]}\underline{R}\vee \frac{\left(s_\mathfrak{a}\int_0^\infty L_a(B^{x,y})dW_\mathfrak{a}(a)\right)^j}{2^jj!}\wedge\overline{R}]dxdy,\end{aligned}
$$
and 
$$\operatorname{Tr}^{(j)}(n;\underline{R},\overline{R})\to\frac{1}{\sqrt{2\pi T}}\int_0^1 \mathbb{E}_{B^{x,x}}[1_{\forall t:B^{x,x}(t)\in [0,1]} \underline{R}\vee \frac{\left(s_\mathfrak{a}\int_0^\infty L_a(B^{x,x})dW_\mathfrak{a}(a)\right)^j}{2^jj!}\wedge \overline{R}]dx.$$

The proof of this Proposition is split into the following steps.

\begin{lemma} Assume $\alpha=\frac{3}{2}$ and Assumption \ref{assumption1.1} holds.
Then proposition \ref{proposition3.5} holds for any $\overline{R}\in[0,\infty)$ and $\underline{R}\in(-\infty,0]$.
\end{lemma}

\begin{proof} It suffices to consider $\operatorname{Sc}^{(j)}$ as the case for $\operatorname{Tr}^{(j)}$ is identical.
    We assume that 
    $$f\geq 0,g\geq 0,\quad\int_0^1 f(x)dx=\int_0^1g(y)dy=1,$$
    as the general case follows by decomposing and re-scaling $f$ and $g$. Let $\lambda,\mu$ be two probability measures on $[0,1]$ with densities $f,g$ respectively. Given any $n\in\mathbb{N}$ and $R\in\mathbb{N}$, consider independent copies $X_r^n,r=1,2,\cdots,R$ of the random walk $X^{\lambda,\mu;n,T_n}$. By Fubini's theorem, the $R$-th moment of $\operatorname{Sc}^{(j)}(n;\underline{R},\overline{R})$ shall be given as, neglecting the $\frac{1}{2}$ multiplicative factor and the integration against $f$ and $g$:
    \begin{equation}
        \begin{aligned}
\mathbb{E}&\Bigg[\prod_{r=1}^R\Xi(X_r^n(0),X_r^n(T_n);n,T_n)1_{\forall t:X_r^n(t)\in [0,n]}\\&\cdot \underline{R}\vee\frac{
\left(\sum_{i'=0}^{T_nn^2}\mathfrak{a}(X_r^n(i'n^{-2}))\right)^j}{j!(2n^{\alpha })^j}\wedge\overline{R}\Bigg].
        \end{aligned}
    \end{equation}
    Now we choose $\alpha=\frac{3}{2}$, then we are in the precise situation to use Proposition \ref{generalization}. Since all the random variables are bounded, the desired convergence in distribution and in moment follow from Proposition \ref{generalization}.
    
\end{proof}

\begin{lemma}\label{587}
    The convergence holds in the sense of moments for any $\underline{R}\in[-\infty,0]$ and $\overline{R}\in[0,\infty]$.
\end{lemma}

\begin{proof}
The convergence in distribution of random variables inside the expectation actually holds if $\overline{R}=\infty$ or $\underline{R}=-\infty$, so we are left to verify the uniform integrability. The second moment of random variable is bounded from above, for any $\overline{R}$ or $\underline{R}$, by 

\begin{equation}
    \mathbb{E}\left[\prod_{r=1}^R \Xi(X_r^n(0),X_r^n(T_n);n,T_n)^2\frac{\left(\sum_{i'=0}^{T_nn^2}\mathfrak{a}(X_r^n(i'n^{-2}))\right)^{2j}}{(j!)^2(2n^\alpha)^{2j}}.
    \right]
\end{equation}
We essentially follow \cite{gorin2018stochastic}, Lemma 4.4 so we only give a sketch.
We bound the first factor $\Xi(X_r^n(0),X_r^n(T_n);n,T_n)^2$ via uniform constants thanks to Lemma \ref{combinatorialbound}, and for the terms involving $\mathfrak{a}$ we use the elementary inequality 
\begin{equation}\label{elementary}
\left(\frac{|z|^j}{j!}\right)^2\leq e^{2|z|}\leq e^{2z}+e^{-2z},\quad z\in\mathbb{R},j\in\mathbb{N}\end{equation}
and the expectations involving $\mathfrak{a}(\cdot)$ can be bounded thanks to Lemma \ref{expoest1} under the Assumption \ref{assumption1.1}. We end up with the upper bound 
\begin{equation}
    C\mathbb{E}\left[ \exp\left( C\sum_{h\in n^{-1}\mathbb{Z}}\left( \frac{\sum_{r=1}^R L_h(X_r^n)^2 }{n}+\frac{\sum_{r=1}^R L_h(X_r^n)^{\gamma'}}{n^{r'/2}}
    \right)
    \right)
    \right]
\end{equation} where we have neglected the term $|\mathbb{E}[\mathfrak{a}(\lfloor nh\rfloor)]|\frac{\sum_{r=1}^R L_h(X_r^n)}{n^{1/2}}$ since $\mathfrak{a}(\cdot)$ has mean zero. Here $2<\gamma'<3$. The desired estimate then follows from the large deviations estimate in Lemma B.2.
\end{proof}

\begin{lemma}
    The convergence holds in distribution for any $\underline{R}\in[-\infty,0]$ and $\overline{R}\in[0,\infty]$.
\end{lemma}

\begin{proof}
The reasoning is exactly the same as \cite{gorin2018stochastic}, Lemma 4.5. For any $\overline{R}\in[0,\infty)$, the variable $\operatorname{Sc}^{(j)}(n;\underline{R},\infty)$ stochastically dominates $\operatorname{Sc}^{(j)}(n;\underline{R},\overline{R})$, so the limit point $\operatorname{Sc}^{(j)}(\infty;\underline{R},\infty)$ stochastically dominate $\lim_{n\to\infty} Sc^{(j)}(n;\underline{R},\overline{R})$. The latter limit tends as $\overline{R}\uparrow\infty$ to the corresponding expression for $\overline{R}=\infty$ by monotone convergence.  Hence $\operatorname{Sc}^{(j)}(\infty;\underline{R},\infty)$ and \eqref{564} with $\overline{R}=\infty$ form two non-negative variables with the same moment, one dominating the other, hence have the same law.
\end{proof}
Combining all these results, we have established Proposition \ref{proposition3.5}.

Now we have obtained the contribution for each power $j$, and now we estimate the remainder term. The results are as follows:

\begin{lemma}\label{momentremainder} Assume $\alpha=\frac{3}{2}$ and we denote by $\operatorname{Sc}^{(j)}(n):=\operatorname{Sc}^{(j)}(n;-\infty,\infty)$ and $\operatorname{Tr}^{(j)}(n):=\operatorname{Tr}^{(j)}(n;-\infty,\infty).$ Assume also Assumption \ref{assumption1.1}. Then
    for any $R\in\mathbb{N}$ we can find constant $C(R)$ such that for each $n\in\mathbb{N}$,
$$\mathbb{E}\left[|\operatorname{Sc}^{(j)}(n)|^R\right]\leq\frac{C(R)}{2^{jR}},\quad \mathbb{E}\left[|\operatorname{Tr}^{(j)}(n)|^R\right]\leq \frac{C(R)}{2^{jR}}.
$$
\end{lemma}

\end{proposition}

\begin{proof} The proof shall be compared to \cite{gorin2018stochastic}, Lemma 4.6 and our case is simpler: the domain is bounded and we don't have the $\zeta(\cdot)$ terms. Expanding the expression of $\mathbb{E}[|\operatorname{Sc}^{(j)}(n)|^R]$ and using Hölder's inequality, we may integrate over the $f$ and $g$ and discard them from the following computations. The function $\Xi$ can be uniformly bounded via Lemma \ref{combinatorialbound}. For the terms involving $|\mathfrak{a}(\cdot)|^R$, we use the inequality \eqref{elementary}, Assumption \ref{assumption1.1} and Lemma \ref{expoest1} in exactly the same way as the proof of Lemma \eqref{587}. Then we are reduced to bounding exponential moments of local times of the following form
\begin{equation}
    \begin{aligned}
        &\sum_{h\in n^{-1}(\mathbb{Z}+\frac{1}{2})}\frac{L_h(X_r^n)}{n^{3/2}},\quad \sum_{h\in n^{-1}(\mathbb{Z}+\frac{1}{2})}\frac{L_h(X_r^n)^2}{n},\\
        & \sum_{h\in n^{-1}(\mathbb{Z}+\frac{1}{2})}\frac{{L_h(X_r^n)}^{\gamma'}}{n^{\gamma'/2}},\quad \sum_{h\in n^{-1}\mathbb{Z}}\frac{L_h(X_r^n)}{n^{3/2}},
        \\&
        \sum_{h\in n^{-1}\mathbb{Z}}\frac{L_h(X_r^n)^2}{n},\quad \sum_{h\in n^{-1}\mathbb{Z}}\frac{L_h(X_r^n)^{\gamma'}}{n^{\gamma'/2}},
    \end{aligned} 
\end{equation}
the first and fourth terms are exactly $T_n n^{-1/2}$. For the remaining terms, we use Proposition \ref{propc2}, note that $\gamma'\in(2,3)$. Thus we have obtained uniform in $n$ exponential moments of these sums of local times, and the proof of the lemma follows.
\end{proof}

We also have the following slight extension:
\begin{lemma}
    Fix $n\in\mathbb{N}$, and consider a random variable $\Gamma$ which is a deterministic function of the random walk bridge $X$ and random variable $\mathfrak{a}(\cdot)$s. Let $\operatorname{Sc}^{(j)}(n;\Gamma)$ denote the modification of $\operatorname{Sc}^{(j)}(n)$ via adding $\Gamma$ as an extra factor in \eqref{564}. Assume $\Gamma$ has $2R$-th moment bounded $\mathcal{C}_1<\infty,$ then we find $C(R)$ such that 
    $$\begin{aligned} \mathbb{E}\left[|\operatorname{Sc}^{(j)}(n;\Gamma)|^R\right]\leq \frac{C(R)}{2^{jR}}C_1^{1/2},\quad \mathbb{E}\left[|\operatorname{Tr}^{(j)}(n;\Gamma)|^R\right]\leq \frac{C(R)}{2^{jR}}C_1^{1/2}.  \end{aligned} $$
\end{lemma}

Before we enter the proof of the main result, we estimate the error of replacing 
$$\sum_{0\leq i_1\leq\cdots\leq i_j\leq\lfloor Tn^2\rfloor-j-\epsilon_{j,x,y}}\prod_{j'=1}^j \mathfrak{a}(X(i_{j'}n^{-2}))$$
by 
$$\left(\sum_{j'=0}^{\lfloor T_nn^2\rfloor-j-\epsilon_{j,x,y}}\mathfrak{a}(X(i_{j'}n^{-2}))\right)^j.$$
This corresponds to Lemma 4.8 of \cite{gorin2018stochastic}. We use the following notations from that work: define
$$h(j;n):=\sum_{0\leq i_1\leq \cdots\leq i_j\leq T_nn^2}\prod_{j'=1}^j\mathfrak{a}(X(i_{j'}n^{-2})),$$
$$p(j;n):=\sum_{i'=0}^{T_nn^2}\mathfrak{a}(X(i'n^{-2}))^j.$$ Expressing $h(j;n)$ in terms of Newton identities (\cite{macdonald1998symmetric}, Chapter 1, Section 2):
\begin{equation}
    h(j;n)=[z^j]\exp\left(\sum_{j'=1}^\infty\frac{p(j',N)}{(j')!}z^{j'}\right)
\end{equation} with $[z^j]$ meaning the coefficient in front of $z^j$ if we expand the exponential as a series.
We make a further expansion and get
\begin{equation}
    \frac{h(j;n)}{(2n^{3/2})^j}=\sum_{\ell=0}^j \frac{p(1;n)^\ell}{l!(2n^{3/2})^\ell}\left([z^{j-\ell}]\exp\left(\sum_{j'=2}^\infty\frac{p(j',n)}{j'(2n^{3/2})^{j'}} z^{j'}\right)
    \right),
\end{equation} 
The case $l=j$ is the leading order contribution, and we will prove contribution from the rest of the terms is small.
\begin{lemma}\label{lemma3.11}
    For each $R=1,2,\cdots$ we may find $\tilde{C}$ and $\epsilon$ positive, so that for any $\ell=1,2,\cdots,T_nn^2$ and $n\in\mathbb{N}$,
    \begin{equation}
\mathbb{E}\left[\left([z^\ell]\exp\left(\sum_{j'=2}^\infty \frac{|p(j';n)|}{j'(2n^{3/2})^{j'}}\right)\right)\right]   \leq (\tilde{C}n^{-\epsilon})^{\ell R}
    \end{equation}

\end{lemma}

\begin{proof}
Defining $\Gamma_j=\sup_{0\leq i_1\leq \cdots\leq i_j}\mathbb{E}[\prod_{j'=1}^j |\mathfrak{a}(i_{j'})|]=\sup_{i\in\mathbb{N}}\mathbb{E}[|\mathfrak{a}(i)|^j]$, which follows from Hölder's inequality. By assumption we have $\Gamma_j\leq C^j j^{j\gamma}$ where $C>0$ is a universal constant and $0<\gamma<\frac{3}{4}$, Hence the right hand side is bounded by 
\begin{equation}
    \frac{1}{(2n^{3/2})^{lR}}\Gamma_{lR}\left([z^l]\exp\left(\sum_{j'=2}^\infty \frac{k}{j'}z^{j'}\right)\right)^R. 
\end{equation}
The rest of the estimate is the same as in the proof of \cite{gorin2018stochastic}, Lemma 4.8, the major difference is 
we replace any $N$ there by $N^3$. We omit the combinatorial details, but note that it follows from the combinatorial inequality (for $l<k$)
$$
[z^l]\exp\left(\sum_{j'=2}^\infty\frac{k}{j'} z^{j'}\right)\leq C^l\left(\frac{k\log l}{l}\right)^{\lfloor l/2\rfloor}.
$$ and that $\gamma<3/4$ by Assumption \ref{assumption1.1}.
\end{proof}

\subsection{Completing the proof of edge scaling limit}
Now we are in the position to complete the proof of Theorem \ref{theorem1.3ups}. We have already computed $(H_n)^{k}[i,i']$
for $k$ is even, and the case for $(H_n)^{k-1}[i,i']$ for $k-1$ is odd is identical, giving rise to an odd number of horizontal segments. Therefore, in order to prove Theorem \ref{theorem1.3ups}, it suffices to show that for any fixed $T>0$, $f,g\in L^2([0,1])$, we have 
\begin{equation}\begin{aligned}
    \operatorname{Sc}(n)\to \frac{1}{\sqrt{2\pi T}}\int_0^\infty\int_0^\infty f(x)g(y)\exp(-\frac{(x-y)^2}{2T})\mathbb{E}_{B^{x,y}}\Bigg[1_{{\forall t:B^{x,y}(t)\in[0,1]}}
    \\
    \exp\left(\frac{\sigma}{2}\int_0^1 L_a(B^{x,y})dW(a)\right)
    \Bigg]dxdy
    \end{aligned}
\end{equation}
and a similar claim for $\operatorname{Tr}(n)$. We only prove the convergence of $\operatorname{Sc}(n)$, as that for $\operatorname{Tr}(n)$ is essentially the same. We have the expansion
\begin{equation}\label{expansion3.27}\begin{aligned}
    \operatorname{Sc}(n)=\sum_{j=0}^{\lfloor Tn^2\rfloor }\operatorname{Sc}^{(j)}(n;-\infty,\infty)+\sum_{j=0}^{\lfloor Tn^2\rfloor}\sum_{l=0}^{j-1}U(n;j,l)
\end{aligned}\end{equation}
where we define \begin{equation}
    \begin{aligned}
U(n;j,l):=&\frac{1}{2}\int_0^1\int_0^1 f(x)g(y)\Xi(x,y;n,T_n)\mathbb{E}_X\Bigg[1_{\{\forall t:0\leq X(t)\leq n\}}
\\ &\frac{\left(\sum_{i'=0}^{T_nn^2}\mathfrak{a}(X(i'n^{-2}))\right)^l}{l!(2n^{3/2})^{l}}\cdot [z^{j-l}]\exp\left(\sum_{j'=2}^\infty \frac{\sum_{i'=0}^{T_nn^2}\mathfrak{a}(X(i'n^{-2}))^{j'}}{j!(2n^{3/2})^{j'}}z^{j'}\right)\Bigg]dxdy.
    \end{aligned}
\end{equation}
Now we take the limit as $n\to\infty$. The limit of the sum of $\operatorname{Sc}^{(j)}(n;-\infty,\infty)$ is identified in Proposition \ref{proposition3.5}. Together with the moment bound of the remainder term (Lemma \ref{momentremainder}), we deduce that the first term in equation \eqref{expansion3.27} converges in the limit $n\to\infty$ towards
\begin{equation}\begin{aligned}
    \frac{1}{\sqrt{2\pi T}}\int_0^1\int_0^1 f(x)g(y(\exp(-\frac{(x-y)^2}{2T}))\mathbb{E}_{B^{x,y}}&\Bigg[1_{B^{x,y}(t)\in[0,1]}\\&\cdot \exp\left(\frac{s_{\mathfrak{a}}}{2}\int_0^\infty L_a(B^{x,y}dW(a))\right)\Bigg]dxdy
\end{aligned}\end{equation} that converges in moment and in law. The last step involves applying the moment decaying estimate for $U(n;j,l)$ given a fixed $R\in\mathbb{N}$. We deduce from Lemma \ref{lemma3.11} that \begin{equation}
    \mathbb{E}[U(n;j,l)^R]\leq \frac{\tilde{C}_1}{2^{lR}}(\tilde{C}_2n^{-\epsilon})^{(j-l)R},\quad j,l\in\mathbb{N},j\geq l
,\end{equation}
given constants $\tilde{C}_1,\tilde{C}_2,\epsilon$ depending on $R$. By maximal inequality for $L^R$ norm, it leads to that the $R$-th moment of the second term in equation \eqref{expansion3.27} tend to 0 in the $n\to\infty$ limit. Since the choice of $R$ is arbitrary, we have justified the convergence in distribution and in moments of Theorem \ref{theorem1.3ups},

\section{Laplace transform of edge scaling limit}

\subsection{Properties of the semigroup}\label{section4.1}
In this section we establish Proposition \ref{proposition1.2}. The proof is split into a number of technical lemmas. As stated in the introduction, these properties can likely be derived from \cite{10.1214/21-EJP654} in a very general setting. Yet we leave the proof here for sake of completeness.

\begin{lemma}\label{prooflemma4.1}For any given $T>0$, $\mathcal{U}(T)$ is Hilbert-Schmidt on $L^2([0,1])$ almost surely.
\end{lemma}

\begin{proof} Fix $T>0$, it suffices to show the integral kernel $K(x,y;T)$ satisfies a.s.
\begin{equation}
    \int_0^1\int_0^1 K(x,y;T)^2dxdy<\infty.
\end{equation}
We indeed show that this holds in expectation. Expanding the expression of $K(x,y;T)$, dropping the indicator function and noting that $[0,1]^2$ is compact, it suffices to bound
\begin{equation}
\mathbb{E}_{B^{x,y}}\left[\exp\left(\frac{\sigma}{2}\int_0^1 L_a(B^{x,y})^2da\right)\right]dxdy,
\end{equation} uniformly for $x,y\in[0,1]$. The exponential moment of Brownian bridge local time is finite thanks to Proposition \ref{propc2}. More precisely, $\int_0^\infty L_a(B^{x,y})^2da$ is the limit of $$\frac{1}{n}\sum_{h\in n^{-1}(\mathbb{Z}_{\geq 0}+\frac{1}{2})}L_h(X^{x,y;n,T_n})^2,$$ and the latter is stochastically dominated by $8T(n^{-1}J_{Tn^2}+n^{-1}\tilde{J}_{Tn^2}+2|x-y|+2n^{-1})$, for $J,\tilde{J}$ the maxima of two independent simple random walks with $Tn^2$ steps of $\pm 1$, by the reasoning in Appendix \ref{appendixC}. Taking the limit, we deduce that $\int_0^\infty L_a(B^{x,y})^2da$ is stochastically dominated by $8T(J+\tilde{J}+2|x-y|),$ for $J,\tilde{J}$ the maximum of two independent Brownian motions on $[0,T]$. This upper bound has finite exponential moment.
\end{proof}

\begin{lemma} $\mathcal{U}(T)$ has the semigroup property: given any $T_1,T_2>0$, we have $\mathcal{U}(T_1)\mathcal{U}(T_2)=\mathcal{U}(T_1+T_2)$ almost surely.
\end{lemma}

\begin{proof}By Fubini's theorem, it suffices to prove that with probability one 
\begin{equation}
    \forall x,y\in\mathbb{R}_{\geq 0}:\quad \int_0^1 K(x,z;T_1)K(z,y;T_2)dz=K(x,y;T_1+T_2).
\end{equation}
The rest of proof is the same as \cite{gorin2018stochastic}, Proposition 2.5 which involves elementary manipulations of conditioned Brownian bridges. We omit the details here.
\end{proof}

\begin{lemma}\label{lemma4.3}
$\mathcal{U}(T)$, $T>0$ is strongly continuous in $L^2$: for any $f\in L^2([0,1])$, $$\lim_{t\to T}\mathbb{E}[\|\mathcal{U}(T)f-\mathcal{U}(t)f\|^2]=0.$$
\end{lemma}
\begin{proof} The argument is similar to, and simpler than \cite{gorin2018stochastic}, Proposition 2.6 so we only give a sketch.
By Cauchy-Schwartz inequality and semigroup property, after some manipulations,
$$\begin{aligned}
\mathbb{E}\|\mathcal{U}(T)f-\mathcal{U}(t)f\|^p\leq \mathbb{E}[\operatorname{Tr}(\mathcal{U})(2p(t\wedge T))]^{1/2}\mathbb{E}[\|\mathcal{U}(|T-t|)f-f\|^{2p}]^{1/2}.
\end{aligned}$$
The term involving the trace of $\mathcal{U}(2p(t\wedge T))$ is bounded for $t$ in a bounded neighborhood, so we only need to prove $\mathbb{E}[\|\mathcal{U}(T)f-f\|^{2p}]\to 0$ as $T\to 0$.

we shall apply the inequality $|e^S-1|\leq |S|e^{|S|}$, to 
$$S:=\frac{\sigma}{2}\int_0^1 L_a(B^{x,y})dW(a),$$ where $B^{x,y}$ is a Brownian bridge on $[0,T]$ from $x$ to $y$, so we need an estimate of  
\begin{equation}
    \begin{aligned}
   \left(\int_0^1\left(\int_0^1 \frac{1}{\sqrt{2\pi T}}\exp(-\frac{(x-y)^2}{2T})\cdot \mathbb{E}_{B^{x,y}}[1_{\forall t:B^{x,y}(t)\in[0,1]}\mathbb{E}_W[|S|^pe^{p|S|}]^{1/p}]|f(y)|dy\right)^2dx\right)^{p/2}.
    \end{aligned}
\end{equation}
We use Cauchy-Schwartz inequality, ignoring the indicator function, noting that the bound on the squared local time has an estimate of the form 
$$\mathbb{E}_{B^{x,y}}[\mathbb{E}_W[|S|^pe^{p|S|}]^{2/p}]^{1/2}\leq CT^{3/4}e^{C|x-y|/\sqrt{T}}$$ thanks to Proposition \ref{propc2} and the discussion at the end of the proof of Lemma \ref{prooflemma4.1} (some further computations are also needed). This upper bound converges to $0$ as $T\to 0$, so we finish the proof of the lemma. 
\end{proof}

\begin{lemma}
We can find an orthogonal basis $\mathbf{v}^1,\mathbf{v}^2,\cdots\in L^2([0,1])$, and $\eta_1\geq\eta_2\geq\cdots$ on the same probability space (both are random) such that the spectrum of $\mathcal{U}(T)$ has eigenvalues $\exp(T\eta^i/2)$ that correspond to eigenvectors $\mathbf{v}^i$, $i\in\mathbb{N}$.    
\end{lemma}

\begin{proof} We follow \cite{gorin2018stochastic}, Proposition 2.4. The symmetry of $\mathcal{U}(T)$, almost surely, follows from the fact that almost surely, 
\begin{equation}
    \forall x,y\in\mathbb{R}_{\geq 0}:\quad K(x,y;T)=K(y,x;T),
\end{equation} which further follows from the expression of $K(\cdot,\cdot,T)$ and that the time reversed Brownian bridge from $x$ to $y$ in time $T$ is a standard Brownian bridge from $y$ to $x$ in time $T$. 

The non-negativity of $\mathcal{U}(T)$ follows directly from definition. Then we know $\mathcal{U}(T)$ is a positive symmetric Hilbert-Schmidt operator almost surely, so we can find an orthogonal basis with eigenvalues $e^1(T)\geq e^2(T)\geq\cdots$. By the semigroup property $\mathcal{U}(T)\mathcal{U}(T/2)=\mathcal{U}(T/2)\mathcal{U}(T),$ so we can find an orthogonal basis of eigenfunctions for both $\mathcal{U}(T)$ and $\mathcal{U}(T/2)$ in $L^2([0,1])$. Then we can simultaneously diagonalize $\mathcal{U}(T)$ and $\mathcal{U}(T/2)$, and this leads to $\sum_{i=1}^\infty e^i(T)=\sum_{i=1}^\infty e^i(T/2)^2$. The latter can be written as 
$$
\int_0^1\int_0^1 K(x,y;T/2)^2dxdy.
$$ As this expression is finite almost surely, this implies $\mathcal{U}(T)$ is trace class. 

We can now complete the proof of Proposition \ref{proposition1.2}. We have checked that for each $T>0$, $\mathcal{U}(T)$ is symmetric and trace class, hence has discrete spectrum. By commutativity of $\mathcal{U}(T)$ almost surely, we can find an orthogonal basis $\mathbf{v}^1,\mathbf{v}^2,\cdots$ of $L^2([0,1])$ that are eigenfunctions of all $\mathcal{U}(T)$ for $T\in(0,\infty)\cap\mathbb{Q}$. Let $e^i$ denote $e^i(1)$ for each $i\in\mathbb{N}$, we order the eigenvectors such that $e^1\geq e^2\geq\cdots$, and we set $\eta^1\geq \eta^2\geq\cdots$ via $\eta_i=2\log e^i,i\in\mathbb{N}$. We claim that none of the eigenvalues $e^i$ of $\mathcal{U}(1)$ will vanish. Note that, if otherwise, then all these operators $\mathcal{U}(T),T\in(0,\infty)\cap\mathbb{Q}$ has $0$ eigenvalue acting on this vector, and this contradicts Lemma \ref{lemma4.3} at $T=0$.

By the semigroup property, $\mathcal{U}(T)\mathbf{v}^i=\exp(T\eta^i/2)\mathbf{v}^i,i\in\mathbb{N}$ for all such $T$, so that $\text{Trace}(\mathcal{U}(T))=\sum_{i=1}^\infty \exp(T\eta^i/2)$. This formula is also true for any $T\in \mathbb{R}_+$ by continuity in $T$ and dominated convergence theorem.

\end{proof}

\subsection{Bounds on extreme eigenvalues}\label{section4.2}

In this section we prove Corollary \ref{corollary1.4}, which gives the Laplace transform of the $n\to\infty$ limiting profile of the edge eigenvalue statistics of the random Schrödinger operator \eqref{randomschrodinger1.1}. The reasoning is similar to \cite{gorin2018stochastic} Lemma 6.1, yet our scaling of eigenvalues is very different from the Gaussian beta ensembles considered in that paper, so we give a sketch of the details.

\begin{lemma} Under the assumptions of Corollary \ref{corollary1.4}
    we have the convergence in distribution
\begin{equation}
    \sum_{i=1}^n e^{T\eta_n^i/2}\to_{n\to\infty}\text{Trace}(\mathcal{U}(T)).
\end{equation}
    simultaneously for finitely many $T$'s.
\end{lemma}

\begin{proof}
    Denoting by $\mu_n^{1,+}\geq \mu_n^{2,+}\geq\cdots$, $\mu_n^{1,-}\leq \mu_n^{2,-}\leq\cdots$ the positive and negative eigenvalues of $H_n$. We work with the rescaled versions
    $$\lambda_n^{i,+}=n^{2}(\mu_n^{i,+}-2),\quad \text{and}$$
    $$\lambda_n^{i,-}=-n^2(\mu_n^{i,-}+2),\quad i=1,2,\cdots.$$
    Then \begin{equation}\label{ofofof}
        \begin{aligned} &\text{Trace}(\mathcal{M}(T,n))
        \\&=\frac{1}{2}\sum_i \left(1+\frac{\lambda_n^{i,+}}{2n^2}\right)^{\lfloor Tn^2\rfloor}+\frac{(-1)^{\lfloor Tn^2\rfloor}}{2}\sum_i \left(1+\frac{\lambda_n^{i,-}}{2n^2}\right)^{\lfloor Tn^2\rfloor}
        \\&+ \frac{1}{2} \sum_i \left(1+\frac{\lambda_n^{i,+}}{2n^2}\right)^{\lfloor Tn^2\rfloor -1}
        \\&+\frac{(-1)^{\lfloor Tn^2\rfloor-1}}{2}\sum_i \left(1+\frac{\lambda_n^{i,-}}{2n^2}\right)^{\lfloor Tn^2\rfloor -1}.
        \end{aligned}
    \end{equation}

Since we already know the convergence of $\text{Trace}(\mathcal{M}(T,n))$ towards $\text{Trace}(\mathcal{U}(T))$, it suffices to show that the difference of the right hand side of \eqref{ofofof} and 
\begin{equation}\begin{aligned}
    \sum_{i=1}^n e^{T\lambda_n^i/2}=&\frac{1}{2}\sum_{i=1}^n e^{T\lambda_n^i/2}+\frac{1}{2}\sum_{i=1}^n e^{T\lambda_n^i/2}+\frac{(-1)^{\lfloor Tn^2\rfloor}}{2}\sum_{i=1}^n e^{T\lambda_n^i/2}\\
    &+\frac{(-1)^{\lfloor Tn^2\rfloor -1}}{2}\sum_{i=1}^n e^{T\lambda_n^i/2}.\end{aligned}
\end{equation}
tends to $0$ as $n$ tends to infinity.

Now we may choose $\epsilon=1/100$.
For this we separate the eigenvalues into four different classes: (1) eigenvalues in the bulk: $\lambda_n^{i,+}$'s and  $\lambda_n^{i,-}$'s less or equal to $-n^\epsilon$;
(2) outlier eigenvalues: $\lambda_n^{i,+}$'s and $\lambda_n^{i,-}$'s greater than $n^\epsilon$; (3) eigenvalues at right edge: $\lambda_n^{i,+}$'s in $(-n^\epsilon,n^\epsilon)$; and (4) eigenvalues at left edge: $\lambda_n^{i,+}$'s in $(-n^\epsilon,n^\epsilon)$.

Then as $n\to\infty$, contribution to the sum of the bulk eigenvalues (1) becomes negligible because there are no more than $n$ of them, with each contributing no more than $e^{-Tn^\epsilon/2}$. 

We then check that as $n\to\infty$, with probability tending to 1 there are no outlier eigenvalues (2). Choose a sequence $T_n,n\in\mathbb{N}$ of positive numbers with $T_nn^2$, $N\in\mathbb{N}$ even integers and $\sup_n |T_n-T|n^2<\infty$. We have
$$
\text{Trace}\left(\left(\frac{H_n}{2}\right)^{T_nn^2}\right)=\sum_i \left(1+\frac{\lambda_n^{i,+}}{2n^2}\right)^{T_nn^2}+\sum_i \left(1+\frac{\lambda_n^{i,-}}{2n^2}\right)^{T_nn^2}
$$Were there an outlier, then \begin{equation}
\text{Trace}\left(\left(\frac{H_n}{2}\right)^{T_nn^2}\right)\geq (1+\frac{n^\epsilon}{2n^2})^{T_nn^2}.
\end{equation}
We have proved that the left hand side converges in distribution as $n\to\infty$ to an almost surely finite limit, but the right hand side tends to infinity, leading to a contradiction. 

Finally we consider the eigenvalues at the edge. We use the approximations 
$$  (1+\frac{\lambda_n^{i,\pm}}{2n^2})^{\lfloor Tn^2\rfloor}\sim e^{T\lambda_n^{i,\pm}/2},\quad (1+\frac{\lambda_n^{i,\pm}}{2n^2})^{\lfloor Tn^2\rfloor -1}\sim e^{T\lambda_n^{i,\pm}/2}.$$ The additive error is upper bounded in norm by 
$$(e^{n^{-2+2\epsilon}/2}-1)(1+\frac{\lambda_n^{i,\pm}}{2n^2})^{\lfloor Tn^2\rfloor}  ,$$
$$  (e^{n^{-2+2\epsilon}/2}-1)(1+\frac{\lambda_n^{i,\pm}}{2n^2})^{\lfloor Tn^2\rfloor-1}  .$$
For this purpose it suffices to prove that 
$$
(e^{n^{-2+2\epsilon}/2}-1)\sum_{i=1}^N(1+\frac{\lambda_n^i}{2n^2})^{\lfloor Tn^2\rfloor-\epsilon},\quad \epsilon\in\{0,1\}
$$
tends to $0$ in probability, which is known from the previous proofs since $e^{n^{-2+2\epsilon}/2}-1$ converges to $0$ and the summation converges to an almost surely finite random variable.

\end{proof}

\section{Tracy-Widom fluctuation for potentials with shifted means}\label{section5}
In this section we outline the procedure to derive Proposition \ref{theorem 1.6}.
It essentially follows from an application of the main result of \cite{ramirez2011beta}, Section 5, but as we will use a similar yet different framework in the forthcoming sections, we choose to give a sketch of the arguments involved for sake of completeness.

Consider a sequence of $\mathbb{R}^2$-valued discrete-time random process $((y_{n,1,k},y_{n,2,k});1\leq k\leq n)$. Given $m_n=o(n)$ a scaling index which we will take to be $m_n=n^{1/3}$. We build an $n\times n$  tridiagonal matrix $\overline{H}_n$ for each $n\in\mathbb{N}_+$:

Consider $T_n$ the shift operator $(T_nv)_k=v_{k+1}$ that maps $\mathbb{R}_1\times\mathbb{R}_2\cdots$ to itself. Consider $(T_n^t v)_k$ its adjoint, denoting by $R_n$ the restriction operator $(R_nv)_k=v_k1_{k\leq n}$. Consider $\Delta_n=m_n(I-T_n^t)$ the operator of difference operator, and set
\begin{equation}\label{barhn}
    \overline{H}_n=R_n\left(-\Delta_n\Delta_n^t+(\Delta_n y_{n,1})_{\times }+(\Delta_ny_{n,2})_{\times }\frac{1}{2}(T_n+T_n^t)\right),
\end{equation}
with the symbol $\times$ standing for element by element multiplication of the corresponding vector. The matrix representation of $\overline{H}_n$ is symmetric tridiagonal, with diagonal elements $(2m_n^2+m_n(y_{n,1,k}-y_{n,1,k-1}),k\geq 1)$ and with elements $(-m_n^2+m_n(y_{n,2,k}-y_{n,2,k-1})/2,k\geq 1)$ up and below the diagonal. Denote further by $y_{n,i}(x):=y_{n,i,\lfloor xm_n\rfloor }1_{xm_n\in[0,n]}$. The following assumptions are made:

\begin{Assumption}\label{assumption2}
     We can find a continuous path $x\to y(x)$ with \begin{equation}
    \begin{aligned}
    (y_{n,i}(x);x\geq 0)\quad i=1,2\text{ are tight in law},\\(y_{n,1}(x)+y_{n,2}(x);x\geq 0)\Rightarrow (y(x);x\geq 0)\text{ in law},
    \end{aligned}
\end{equation}
for the Skorokhod topology.
\end{Assumption}

\begin{Assumption}\label{assumption3}  We have a decomposition
\begin{equation}
    y_{n,i,k}=m_n^{-1}\sum_{\ell=1}^k \eta_{n,i,\ell}+w_{n,i,k},
\end{equation} given $\eta_{n,i,k}\geq 0$, and for deterministic non-decreasing function $\bar{\eta}(x)>0,\zeta(x)\geq 1$, and random constants $\kappa_n(\omega)\geq 1$ on the same probability space that satisfy: the random constants $\kappa_n$ are tight in law, and
\begin{equation}
    \bar{\eta}(x)/\kappa_n-\kappa_n\leq \eta_{n,1}(x)+\eta_{n,2}(x)\leq \kappa_n(1+\bar{\eta}(x)),
\end{equation}
\begin{equation}
    \eta_{n,2}(x)\leq 2m_n^2,
\end{equation}
\begin{equation}\label{moment}
    |w_{n,1}(\zeta)-w_{n,1}(x)|^2+|w_{n,2}(\zeta)-w_{n,2}(x)|^2\leq\kappa_n(1+\bar{\eta}(x)/\zeta(x)),
\end{equation}given any $n$ and $x,\zeta\in[0,n/m_n]$ such that $|x-\zeta|\leq 1$.
\end{Assumption}

Consider the random Schrödinger operator
\begin{equation}\label{randomschrodinger}
    H=-\frac{d^2}{dx^2}+y'(x)
\end{equation} which maps $H_{loc}^1$ to $\mathcal{D}$, the space of distributions on $[0,\infty)$. More precisely, the operator is defined on the Hilbert space $L^*\subset L^2(\mathbb{R}^+)$ which consists of functions satisfying $f(0)=0$, and
$$\|f\|_*^2:=\int_0^\infty f'(x)^2+(1+\bar{\eta}(x))f^2(x)dx<\infty.$$ We say that $(\lambda,f)\in\mathbb{R}\times L^*\setminus\{0\}$ with $\|f\|_2=1 $ is an eigenvalue of \eqref{randomschrodinger} if it satisfies \eqref{randomschrodinger}, or equivalently, for any $\varphi\in \mathcal{C}_0^\infty,$ we have
$$ \int f\varphi''dx=\int -\lambda f\varphi-yf'\varphi-yf\varphi'dx.$$ 

Then the main convergence result from  \cite{ramirez2011beta}, Chapter 5 is:

\begin{theorem}\label{5.35.3}[Theorem 5.1 of \cite{ramirez2011beta}] Under Assumptions \ref{assumption2} and \ref{assumption3}, given any fixed $k$, then the smallest $k$ eigenvalues of $\overline{H}_n$ converge to the smallest $k$ eigenvalues of the random Schrödinger operator $H$.
    
\end{theorem}
Denote by $\Delta y_{n,k}=y_{n,k}-y_{n,k-1}$, then by Corollary 6.1 of \cite{ramirez2011beta}, assuming that for some $a\in\mathbb{R}$ and $h\in\mathcal{C}_1(\mathbb{R}^+)$, the random process $y_n$ with $y_{n,0}=0$ has independent increments and satisfy
$$ m_n\mathbb{E}\Delta y_{n,k}=h'(k/m_n)+o(1),\quad m_n\mathbb{E}
(\Delta y_{n,k})^2=a^2+o(1),\quad m_n\mathbb{E}(\Delta y_{n,k})^4=o(1),$$
uniformly for $k/m_n$ on compacts of $(0,\infty)$, then $y_n(t)=y_{n,\lfloor tm_n\rfloor}$ converges in distribution to $h(t)+ab_t$, for $b$ the standard Brownian motion, with respect to the Skorokhod topology.

For the Gaussian $\beta$-ensemble, thanks to the tridiagonal matrix model introduced in \cite{dumitriu2002matrix}, in \cite{ramirez2011beta} the authors chose random potentials to be 
$$y_{n,1,k}=-n^{-1/6}(2/\beta)^{1/2}\sum_{\ell=1}^kg_\ell,
$$
$$y_{n,2,k}=n^{-1/6}\sum_{\ell=1}^k2(\sqrt{n}-\frac{1}{\sqrt{\beta}}\chi_{\beta(n-\ell)}).
$$ where $g_{\ell}$ are independent Gaussians with variance 2 and $\chi_\cdot $ are the chi-distributions with parameter specified in the subscript. Then by \cite{ramirez2011beta}, Lemma 6.2, there is convergence in Skorokhod topology 
$y_{n,i}(\cdot)\Rightarrow (2/\beta)^{1/2}b_x+\frac{x^2}{2}(i-1), i=1,2.$ This proves that the edge fluctuations of the Gaussian $\beta$-ensemble are governed by the Tracy-Widom $(\beta)$- distribution, which for general values of $\beta$ are also defined in \cite{ramirez2011beta}.

For our random Schrödinger operator $H_n^\beta$ defined in \eqref{randomschrodingerT.1} of Proposition \ref{theorem 1.6}, we instead take, in the special case $m_n=n^{1/3}$, 
$$y_{n,1,k}=n^{1/3}\sum_{\ell=1}^k \left(\frac{\ell}{n} -\frac{2}{\sqrt{\beta}}\frac{1}{\sqrt{n}}\mathfrak{a}(\ell)\right).$$
$$y_{n,2,k}=0.$$
and in the general case $m_n=o(n)$, take $$y_{n,1,k}=m_n\sum_{\ell=1}^k \left(\frac{\ell}{(m_n)^3} -\frac{2}{\sqrt{\beta}}\frac{1}{(m_n)^{3/2}}\mathfrak{a}(\ell)\right).$$
Then in all these cases
$$y_{n,1}(\cdot)\Rightarrow -\frac{2}{\sqrt{\beta}}b_x+\frac{x^2}{2}$$
and $$y_{n,2}(\cdot)\Rightarrow 0.$$ As in the case of $\beta$-ensembles, we see that Assumption \ref{assumption3} holds with $\bar{\eta}(x)=x$, and it is not hard to check that the moment condition \ref{moment} is also satisfied. In our case $y(x)=\frac{2}{\sqrt{\beta}}b_x+\frac{x^2}{2},$ so by Theorem \ref{5.35.3}, our limiting object is the random Schrödinger operator $H$ \eqref{randomschrodinger} with integrated potential $y$, and this is exactly the stochastic Airy operator $\mathcal{H}_\beta$ defined in \eqref{starbeta}. Combining all these discussions, we have proved $H_n^\beta$ has fluctuation at the top edge described by Tracy-Widom $\beta$-distribution as $n\to\infty$. This completes the proof of Proposition \ref{theorem 1.6}.

\section{Schrödinger operator interpretation of edge scaling limit}
In this section we investigate properties of the random Schrödinger operator with zero boundary condition
\begin{equation}
    \mathcal{G}_\sigma:=-\frac{d^2}{dx^2}+\sigma b_x',\quad x\in[0,1],
\end{equation}
and prove that the rescaled largest eigenvalues of the random Schrödinger operator $H_n$ \eqref{randomschrodinger1.1} converge to that of $\mathcal{G}_\sigma$ as $n\to\infty$.

\subsection{Riccati transform}\label{riccatis}
In this section we recall some properties of the Riccati transform corresponding to $\mathcal{G}_\sigma$ that will be used later.

Recall that $(\lambda,\varphi)$ is an eigenvalue-eigenvector pair to $\mathcal{G}_\sigma$ if we have 
\begin{equation}
    \varphi''(x)=\sigma \varphi(x)b_x'-\lambda \varphi(x).
\end{equation}
Now we set $p(x)=\varphi'(x)/\varphi(x)$, so that $p(x)$ satisfies $p(0)=\infty$ and solves the SDE
\begin{equation}\label{toendoce}
    p'(x)=-\lambda-p^2(x)-\sigma b'(x).
\end{equation}

To encode the dependence of $p$ on $\lambda$, we may also write $p(x,\lambda)$ for $p(x)$.
To account for the blowup of $p(x)$ to $-\infty$, we assume that every time $p(x)$ reaches $-\infty$, it immediately restarts at $+\infty$. As in \cite{ramirez2011beta}, we may consider $p$ taking values in a countable disjoint union of reals $\mathbb{R}_0,\mathbb{R}_{-1},\mathbb{R}_{-2},\cdots$. We order points $(n,x)$ in lexicographic order and endow the topology of two-point compactification of each copy of $\mathbb{R}$, where we glue the endpoints following the lexicographic order, that is, we glue up $(n,-\infty)$ and $(n-1,+\infty)$ for each $n\in\mathbb{Z}_{\leq 0}$.

Our argument is based on the following lemma (see also \cite{ramirez2011beta},Lemma 3.2):

\begin{lemma}\label{lemma3.5}
    For fixed $\lambda$, denote by $(-n,y)=p(1,\lambda)$. Then the total number $n$ of blow-ups of $p(x)$ to $-\infty$ for $x\in[0,1]$ equals the number of eigenvalues of $\mathcal{G}_\sigma$ in $(-\infty,\lambda]$.
\end{lemma}

\begin{proof}
    By definition, $\lambda$ is an eigenvalue of $\mathcal{G}_\sigma$ is equivalent to $p_\lambda$ (solution to the SDE \eqref{toendoce} with parameter $\lambda$) has a blow up at $x=1$ to $-\infty$. Almost surely, for $\lambda$ sufficiently negative no blowup of $p_\lambda$ occurs on $[0,1]$. Increasing $\lambda$, then continuity and monotonicity
 pushes the existing blowups to the start of the interval and new blow ups emerge near $x=1$. Each such $\lambda$ where a new blow up of $p_\lambda$ occurs corresponds to a new eigenvalue.
\end{proof}

The properties of $\mathcal{G}_\sigma$ stated in Proposition \ref{proposition1.9} can be derived with the help of this Riccati transform $p(x)$. But as the conclusions of Proposition \ref{proposition1.9} have essentially been covered by \cite{Fukushima1977OnSO} and \cite{doi:10.1142/S0219199715500820}, we choose to omit the proof.

\subsection{The random operator as edge scaling limit}\label{sec6.3nomoment}
In this section we prove that the random Schrödinger operator $\mathcal{G}_\sigma$
precisely describes the edge scaling limit of $H_n$ defined in \eqref{randomschrodinger1.1}. Our proof uses the same notation as in Section \ref{section5}, and is an adaptation of the proof of \cite{ramirez2011beta}, Section 5. Before going to the details, we show how our proof is similar to, and different from \cite{ramirez2011beta}: the main idea of \cite{ramirez2011beta} is that if the on-diagonal and off-diagonal potentials (of the rescaled and re-centered matrix $\bar{H}_n$ in \eqref{barhn}) sum up to the potential of the Schrödinger operator $\mathcal{H}_\beta$, then the eigenvalues of that matrix converge to eigenvalues of $\mathcal{H}_\beta$. We use the same idea here. However, in \cite{ramirez2011beta} where the limiting law is Tracy-Widom, we need the potentials have a deterministic slope (see the function $\bar{\eta}(x)>0$ in Assumption \ref{assumption3}) so that resulting operator $\mathcal{H}_\beta$ defined on $[0,\infty)$ has eigenvalues bounded from below, giving rise to a discrete spectrum. This corresponds to the $+x$ term in the definition of $\mathcal{H}_\beta$, \eqref{starbeta}. In our case, we don't have this deterministic slope and the potentials have mean $0$. This does not cause any issue because our Schrödinger operator $\mathcal{G}_\sigma$ is defined on the compact interval $[0,1]$, hence a priori has a discrete spectrum almost surely.

As in Section \ref{section5}, we consider the rescaled matrix 
\begin{equation}\label{barhn5}
    \overline{H}_n=R_n\left(-\Delta_n\Delta_n^t+(\Delta_n y_{n,1})_{\times }+(\Delta_ny_{n,2})_{\times }\frac{1}{2}(T_n+T_n^t)\right),
\end{equation}
but now we take $m_n=n$. We take $y_{n,2,k}=0$ for each $n$, $k$ and we take 
\begin{equation}
    y_{n,1,k}=-n^{-1/2}\sum_{\ell=1}^k \sigma\mathfrak{a}(\ell)=w_{n,1,k}.
\end{equation}
With this choice, the matrix $\overline{H}_n$ is tridiagonal, with $-n^2$ above and below the diagonal and $2n^2-n^{1/2}\sigma\mathfrak{a}(\ell)$ on the diagonal. Thus $\overline{H}_n=-n^2(H_n-2I_n)$ where $H_n$ is the matrix representation \eqref{matrixreps} of \eqref{randomschrodinger1.1}.

Now we define $y_{n,1}(x)=y_{n,1,\lfloor xn\rfloor}1_{xn\in[0,n]}$, we have the convergence of \begin{equation}\label{6.6!}
    y_{n,1}(x),x\in[0,1]\Rightarrow \sigma b(x),x\in[0,1]
\end{equation}
in law, with respect to the Skorokhod topology on paths, where $b$ is a standard Brownian motion with $b(0)=0.$

Define also $w_{n,1}(x)=w_{n,1,\lfloor xn\rfloor}1_{xn\in[0,n]}$, the following estimate is useful:
\begin{lemma}\label{lemmatight}
    we can find a tight sequence of random variables $\kappa_n$ such that 
    \begin{equation}\label{boundsusp}
        |w_{n,1}(\zeta)-w_{n,1}(0)|^2\leq\kappa_n
    \end{equation}
    for any $n\in\mathbb{N}_+$, and any $\zeta\in[0,1]$.
\end{lemma}

\begin{proof}
To show tightness of 
$$\sup_{\ell=0,1,\cdots,n}|w_{n,1,\ell}-w_{n,1,0}|^2$$
we use the $L^p$ maximal inequality of martingales to deduce that 
$$
\mathbb{E}\sup_{\ell=0,\cdots,n}|w_{n,1,\ell}-w_{n,1,0}|^4\leq 16\mathbb{E}|w_{n,1,n}-w_{n,1,0}|^4<C<\infty,
$$ by the moment assumptions on $\mathfrak{a}(\ell)$ in Assumption \ref{assumption1.1}. Indeed, a uniform fourth moment bound on $\mathfrak{a}(\cdot)$ suffices. 
\end{proof}

Since the upper bound $\kappa_n$ in Lemma \ref{lemmatight} is tight, to prove the convergence in Theorem, it suffices to prove convergence for deterministic coefficients (as in \cite{ramirez2011beta}, Proposition 5.2). Then the remaining part of this section is to prove:

\begin{proposition}[Deterministic convergence] We assume convergence \eqref{6.6!} holds in a deterministic way and the bound \eqref{boundsusp} holds for a deterministic constant $\kappa$. Then given any $k\in\mathbb{N}_+$, the smallest $k$ eigenvalues of $\overline{H}_n$ converge to the lowest $k$ eigenvalues of $\mathcal{G}_\sigma$.
\end{proposition}

To deduce tightness, consider a discrete version of the norm $\|\cdot\|_*$ as follows: for $v\in\mathbb{R}^n$, define a norm $\|v\|_2^2:=n^{-1}\sum_{k=1}^n v_k^2$. Recall that we define the difference quotient $\Delta_n=n(I-T_n^t)$, then we set the norm 
\begin{equation}
    \|v\|_{*n}^2:=\|\Delta_n v\|_2^2+\|v\|_2^2.
\end{equation}
We prove the following bound on $H_n$:
\begin{lemma}\label{lemma6.666}
    We can find $c_{11},c_{12},c_{13}>0$ such that for any $n$ and $v$ we obtain 
    $$c_{11}\|v\|_{*n}^2-c_{12}\|v\|_2^2\leq\langle v,\overline{H}_nv\rangle\leq c_{13}\|v\|_{*n}^2.$$
\end{lemma}
\begin{proof}
 By definition of $\overline{H}_n$, let $w_k=w_{1,k}$, recall by definition $\Delta v_k=n(v_{k+1}-v_k)$,
 \begin{equation}
     n\langle v,\overline{H}_nv\rangle=\sum_{k=0}^n(\Delta v_k)^2+\sum_{k=0}^n(\Delta w_k v_k^2),
 \end{equation}
Write $A$, $B$ the two sums. Then $A=n|\Delta_nv|_2^2$. To estimate $B$, we rearrange the order of summation by
\begin{equation}
    \sum_{k=0}^n \Delta w_kv_k^2=n\sum_{k=0}^n w_k (v_{k-1}^2-v_k^2)
\end{equation} where $v_{-1}=0$.
Since by assumption $w_k$ is uniformly bounded, we use the elementary inequality 
$$2n|v_{k+1}^2-v_k^2|\leq pn^2|v_{k+1}-v_k|^2+\frac{1}{p}|v_{k+1}+v_k|^2\leq p|\Delta v_k|^2+\frac{2}{p}(|v_k|^2+|v_{k+1}|^2)$$ and by choosing $p$ sufficiently small to deduce the resulting estimate, so that in the lower bound the coefficient in front of $\|v\|_{*n}^2$ is positive.\end{proof}

Now we prove the operator convergence. Take an embedding of the domain $\mathbb{R}^n$ of $\overline{H}_n$ in $L^2([0,1])$ isometrically via identifying $v\in\mathbb{R}^n$ with step function $v(x)=v_{\lfloor nx\rfloor}$ supported on $[0,1]$. Consider $L_n^*\in L^2([0,1])$ the space consisting of these step functions, and denote by $\mathcal{P}_n$ the $L^2$ projection from $L^2([0,1])$ onto it. Denote by $(T_nf)(x)=f(x+n^{-1})$ the shift operator, and $R_n(f)=f1_{[0,1]}$ the restriction operator. Then we have the following properties: (i) $\mathcal{P}_n$, $T_n$ and $\Delta_n$ commute; (ii) given $f\in L^2$ then $\mathcal{P}_nf\to f$ in $L^2$; (iii) given $f'\in L^2$ and $f(0)=0$ there is $\Delta_n f\to f'$ in $L^2$.

\begin{lemma}\label{lemma6.6}
    Suppose $f_n\in L_n^*$ with $f_n\to f$ in $L^2$ weakly, and moreover $\Delta_n f_n\to f'$ in $L^2$ weakly. Then given any $\varphi\in \mathcal{C}_0^\infty([0,1])$ we have
    $\langle \varphi,\overline{H}_nf_n\rangle\to\langle \varphi,\mathcal{G}_\sigma f\rangle$. Further,
    \begin{equation}\label{514514}
        \langle \mathcal{P}_n\varphi,\overline{H}_n\mathcal{P}_n\varphi\rangle=\langle \varphi,\overline{H}_n\mathcal{P}_n\varphi\rangle\to\langle \varphi,\mathcal{G}_\sigma\varphi\rangle.
    \end{equation}
\end{lemma}

\begin{proof} We adapt the proof of \cite{ramirez2011beta}, Lemma 5.7.
    We assume that $\varphi$ is supported in $(0,1)$, hence drop $R_n$ from $H_n$. The convergence for the free Laplacian part
    $$\langle \varphi,\Delta_n \Delta_n^tf\rangle =\langle \Delta_n\Delta_n^t\varphi,f\rangle\to \langle \varphi'',f\rangle=\langle \varphi,f''\rangle$$ is self-evident, so we only need to check the potential term. Note that if $I\subset[0,1]$ is a finite interval, $g_n\to_{L^2}g$ strongly and $h_n\to h$ converges weakly and the sequence is bounded in $L^2(I)$ then 
    \begin{equation}\label{weakstars}\langle g_n,h_n1_I\rangle \to \langle g,h1_I\rangle.\end{equation}
    Now the potential term is
$$\langle \varphi,({(\Delta_n y_{n,1})}_{\times}+{(\Delta_ny_{n,2})}_{\times}\frac{1}{2}(T_n+T_n^t))f\rangle$$
Note that we have  no $y_{n,2}$ terms. We write $y_n=y_{n,1}$ and approximate the right by 
$$
\langle \varphi,(\Delta_n y_n)_{\times}f_n\rangle=\langle \Delta_n^tf_n,\varphi y_n\rangle+\langle f_n,y_n\Delta_n^t\varphi\rangle+n^{-1}\langle\Delta_n^tf_n,y_n\Delta_n^t\varphi\rangle.
$$
Now the first two terms converge to the expected limits thanks to \eqref{weakstars} and \eqref{6.6!} and the last term converges to $0$ as $n\to\infty$.

Since in our case there is no $y_{n,2}$ terms, we don't need the error estimate as in \cite{ramirez2011beta}, equation (5.16). This completes the proof.
\end{proof}

\begin{lemma}\label{lemma6.888}
    Given any sequence $f_n\in L_n^*$ with $\|f_n\|_{*n}\leq c$ and $\|f_n\|_2=1,$ then we have a $f\in L^*$ and subsequence $n_k$ with $f_{n_k}\to_{L^2}f$, and for any $\varphi\in\mathcal{C}_0^\infty$, there is $\langle \varphi,\overline{H}_{n_k}f_{n_k}\rangle\to \langle \varphi,\mathcal{G}_\sigma f\rangle.$
\end{lemma}

\begin{proof}
As $f_n$, $\Delta_n f_n$ are bounded in $L^2$, upon taking a subsequence we have $f_n\to f$ weakly in $L^2$, and $\Delta_nf_n\to\tilde{f}$ weakly in $L^2$. Taking $\varphi=1_{[0,t]}$ shows $f$ is differentiable with $f'=\tilde{f}$. By lower semi-continuity $f\in L^*$, and this provides enough tightness for us to deduce that $f_n\to f$ strongly in $L^2$. The rest follows from Lemma \ref{lemma6.6}.
\end{proof}

We are in the position to complete the proof of Theorem \ref{jamstheorem}. We introduce the following notations that will only be used in the following two lemmas: let $(\lambda_{n,k},v_{n,k}),k\geq 0$ be the smallest eigenvalues and normalized embedded eigenfunctions of $\overline{H}_n$, and let $(\Lambda_k,f_k)$ be the same for $\mathcal{G}_\sigma$.

The proof of the following lemma is the same as \cite{ramirez2011beta}, Lemma 5.9.
\begin{lemma}
    For any $k\geq 0$ we have 
    $\underline{\lambda}_k=\lim\inf_n \lambda_{n,k}\geq\Lambda_k$.
\end{lemma}

\begin{proof}
The eigenvalues of $\overline{H}_n$ are uniformly bounded from below, thus for some subsequence, it converges to a limit$$(\lambda_{n,1},\cdots,\lambda_{n,k})\to (\zeta_1,\cdots,\zeta_k=\underline{\lambda}_k).$$ The eigenfunctions corresponding to these eigenvalues have bounded $L_n^*$ norm thanks to Lemma \ref{lemma6.666}, so by Lemma \ref{lemma6.888}, along a further subsequence, these eigenfunctions converge in $L^2$. Moreover, the limiting eigenfunctions must be orthogonal as well, so they correspond to $k$ distinct states.
\end{proof}

The last remaining step of the proof of Theorem \ref{jamstheorem} is the following lemma:

\begin{lemma} For each $k\geq 0$ the convergence $\lambda_{n,k}\to\Lambda_k$ and $v_{n,k}\to_{L^2}f_k$ holds.
\end{lemma}

\begin{proof}
The proof is the same as \cite{ramirez2011beta}, Lemma 5.10. Assume by induction the claim is verified up to $k-1$. Choose a $f_k^\epsilon\in\mathcal{C}_0^\infty([0,1])$ that is $\epsilon$-close to $f_k$ in $L^*$. Define a vector\begin{equation}
f_{n,k}=\mathcal{P}_nf_k^\epsilon-\sum_{\ell=1}^{k-1}\langle v_{n,\ell},\mathcal{P}_nf_k^\epsilon\rangle v_{n,\ell}.
\end{equation}
A uniform bound for the $L_n^*$ norm of $v_{n,\ell}$ follows from Lemma \ref{lemma6.666}, and $\|\mathcal{P}_n f_k^\epsilon-v_{n,k}\|_2\leq 2\epsilon$ for $n$ large, so that the $L_n^*$ norm of the summation is no more than $c\epsilon$. We have
\begin{equation}
    \lim\sup_{n\to\infty}\lambda_{n,k}\leq\lim\sup_{n\to\infty}\frac{\langle f_{n,k},\overline{H}_nf_{n,k}\rangle}{\langle f_{n,k},f_{n,k}\rangle}=\lim\sup_{n\to\infty}\frac{\langle\mathcal{P}_nf_k^\epsilon,\overline{H}_n\mathcal{P}_nf_k^\epsilon\rangle}{\langle \mathcal{P}_nf_k^\epsilon,\mathcal{P}_nf_k^\epsilon\rangle }+o_\epsilon(1).
\end{equation}
By \eqref{514514}, $\lim_{n\to\infty}\langle \mathcal{P}_nf_k^\epsilon,\overline{H}_n\mathcal{P}_nf_k^\epsilon\rangle=\langle f_k^\epsilon, \mathcal{G}_\sigma f_k^\epsilon\rangle,$ so that the right hand side equals 
$$\frac{\langle f_k^\epsilon,\mathcal{G}_\sigma f_k^\epsilon\rangle }{\langle f_k^\epsilon,f_k^\epsilon\rangle}+o_\epsilon(1)=\frac{\langle f_k,\mathcal{G}_\sigma f_k\rangle}{\langle f_k,f_k\rangle}+o_\epsilon(1).$$
Finally, in the $\epsilon\to 0$ limit, the right converges to $\langle f_k,\mathcal{G}_\sigma f_k\rangle/\langle f_k,f_k\rangle=\Lambda_k$. Further, for a subsequence of $v_{n,k}$ we find a further subsequence converging strongly in $L^2$ to $g\in L^*$. The $g$ satisfies $\mathcal{G}_\sigma g=\Lambda_k g$ in distribution so $g=f_k$ and $v_{n,k}\to_{L^2} f_k.$ This completes the proof.
\end{proof}

\subsection{Equivalence of operators}\label{section6.4}
In this section we prove Corollary \ref{corollary1.11}, that is, the equivalence of $e^{-\frac{T}{2}\mathcal{G}_\sigma}$ and $\mathcal{U}(T)$ under the given coupling of Brownian motion $W$. We do not claim any originality for this result as it seems to be covered by the very recent paper \cite{10.1214/21-EJP654}, yet we keep the proof as it is very short.

\begin{proof}
We follow the proof of \cite{gorin2018stochastic}, Corollary 2.2. As detailed in Section \ref{sec6.3nomoment}, we identify the matrix $H_n$ as an operator on $L^2([0,1])$ by interpreting $H_n(\pi_nf)$ as a piecewise constant function on the intervals $[0,n^{-1}),[n^{-1},2n^{-1}),\cdots,[1-n^{-1},1)$, which takes values as $n^{1/2}$ multiples of the value of $H_n(\pi_nf).$ We have detailed in Section \ref{sec6.3nomoment} a coupling of $H_n,n\in\mathbb{N}$ whose scaled largest eigenvalues and eigenvectors converge to that of $-\frac{1}{2}\mathcal{G}_\sigma$ almost surely, with the Brownian motion $W$ arising as we take the limit \eqref{208}. 

Under that coupling, the largest eigenvalues and associated eigenvectors of $\mathcal{M}(T,n),n\in\mathbb{N}$ converge to that of $e^{-\frac{T}{2}\mathcal{G}_\sigma}$ with probability one. For the eigenvectors, this is because $\mathcal{M}(T,n)$, $e^{-\frac{T}{2}\mathcal{G}_\sigma}$, and $-\frac{1}{2}\mathcal{G}_\sigma$ have the same eigenvectors, and convergence of eigenvalues follow from approximating exponential of eigenvalues of $\mathcal{G}_\sigma$ by high powers of eigenvalues of $\mathcal{G}_\sigma$ just as we did in Section \ref{section4.2}, and use the convergence results in that section.

Since eigenvalues of $-\frac{1}{2}\mathcal{G}_\sigma$ converge almost surely to $-\infty$ thanks to Proposition \ref{proposition1.9}, we have the a.s. strong convergence of matrices $\mathcal{M}(T,n),n\in\mathbb{N}$ (as operators on $L^2([0,1])$ to $e^{-\frac{T}{2}\mathcal{G}_\sigma}$. This implies weak convergence in terms of finite dimensional distributions, so that $\int_0^1 f(x)(e^{-\frac{T}{2}\mathcal{G}_\sigma}g)(x)dx$, $f,g\in L^2([0,1]),$ and $\int_0^1 f(x)(\mathcal{U}(T)g)(x)$,$f,g\in L^2([0,1]),$ are identical, and their joint distributions, coupled with the law of the Brownian motion $W$ that appears in the definition of $\mathcal{G}_\sigma$ and $\mathcal{U}(T)$, are equal.
That is, the pair of laws $(W,e^{-\frac{T}{2}\mathcal{G}_\sigma})$ and $(W,\mathcal{U}(T))$ are identical, so that we can find a unique, up to set of measure $0$, deterministic function $F$ for which $(W,F(W))$ has the same law as $(W,e^{-\frac{T}{2}\mathcal{G}_\sigma})$ and $(W,\mathcal{U}(T))$. Thus, via the appropriate coupling of $W$, we identify $e^{-\frac{T}{2}\mathcal{G}_\sigma}$ and $\mathcal{U}(T)$.

\end{proof}

\section{Tail estimates of top eigenvalue}\label{section7}
In this section we prove Theorem \ref{theorem1.11}.
With $\operatorname{RSO}_\sigma=-\Lambda_0(\sigma)$ where $\Lambda_0(\sigma)$ is the smallest eigenvalue of $\mathcal{G}_\sigma$, we can use the variational characterization of $\Lambda_0(\sigma)$ and the Riccati transform $p(x)$ \eqref{toendoce} to deduce upper and lower tail estimates of $\operatorname{RSO}_\sigma$. We follow similar arguments as in \cite{ramirez2011beta},Section 4 to derive left and right tail estimates up to the first order. Indeed, in our case $\operatorname{RSO}_\sigma$ has the same  right tail asymptotic as $\operatorname{TW}_\beta$,
yet the left tail asymptotic of $\operatorname{RSO}_\sigma$ is very different from that of $\operatorname{TW}_\beta$.

\textbf{Right tail, lower bound}
Observe that
\begin{equation}
    \begin{aligned}
    \mathbb{P}(\operatorname{RSO}_\sigma>a)=\mathbb{P}(\Lambda_0(\sigma)<-a)&\geq \mathbb{P}(\langle f,\mathcal{G}_\sigma f\rangle<-a\langle f,f,\rangle)\\&=\mathbb{P}(\sigma \|f\|_4^2\mathfrak{g}<-a\|f\|_2^2-\|f'\|_2^2)
    \end{aligned} 
\end{equation}
given any $f\in L^*$, and $\mathfrak{g}$ is the standard Gaussian.
Now we take $f(x)=\sech(\sqrt{a}(x-\frac{1}{2}))$.  Then with $\sim$ denotes the $a\uparrow \infty$ asymptotic, we have
the following asymptotic (with the norm $\|\cdot\|$ taken on $\mathbb{R}$):
\begin{equation}\label{sharpraes}a\|f\|_2^2\sim 2\sqrt{a},\quad \|f'\|_2^2\sim \frac{2}{3}\sqrt{a},\quad \|f\|_4^4\sim \frac{4}{3\sqrt{a}}.
\end{equation}
While $f$ does not satisfy the boundary condition  $f(0)=f(1)=0,$ their values decay exponentially fast as $a$ increases, and outside $[0,1]$ $f$ can be bounded by an exponentially vanishing function whose $L^2,$ $L^4$ norm and $L^2$ norm of its derivative (all the norms are with respect to  $L^2(\mathbb{R}\setminus[0,1])$) are all negligible compared to the estimate \eqref{sharpraes}. So upon a slight modification of $f$ , we may assume $f$ satisfies the boundary conditions $f(0)=f(1)=0$ and $f$ satisfies estimate \eqref{sharpraes} where all the norms are taken on $L^2([0,1]).$

Therefore 
$$\mathbb{P}(\operatorname{RSO}_\sigma>a)\geq \mathbb{P}\left(\sigma \times \frac{2}{\sqrt{3}}a^{-1/4}\mathfrak{g}<-a^{1/2}(2+\frac{2}{3}+o(1))\right),
$$
so by properties of Gaussian $\mathbb{P}(\mathfrak{g}>c)=e
^{-c^2(\frac{1}{2}+o(a))}$ we deduce that for $a\uparrow\infty,$
$$
\mathbb{P}(\operatorname{RSO}_\sigma>a)\geq 
e^{-\frac{8}{3\sigma^2}a^{3/2}(1+o(1))}.
$$

\textbf{Left tail, upper bound}
We use the same idea but this time choose, for sufficiently large $a$,
$$f_a(x)=\begin{cases} a^{0.1}x,\quad x\in[0,a^{-0.1}],\\ 1,\quad x\in [a^{-0.1},1-a^{-0.1}],\\ a^{0.1}(1-x),\quad x\in [1-a^{-0.1},1].
\end{cases}$$
Then we have the $a\uparrow\infty$ asymptotic
$$ a\|f\|_2^2\sim a,
,\quad \|f'\|_2^2=o(a^{0.5}), \quad \|f\|_4^4\sim 1 .
$$
Therefore 
$$
\mathbb{P}(\operatorname{RSO_\sigma}<-a)\leq\mathbb{P}\left(\sigma \mathfrak{g}> a\right)=e^{-\frac{a^2}{2\sigma^2}(1+o(1))}.
$$

\textbf{Left tail, lower bound}
By the diffusion description, $$\mathbb{P}(\operatorname{RSO}_\sigma<-a)=\mathbb{P}_\infty(p_a\text{ does not explode on  [0,1]})$$ where $p_a$ solves the SDE
 $$p'(x)=-a-p^2(x)-\sigma b'(x)$$
with $p_a(0)=+\infty$. By monotonicity, this probability is larger than the probability that $p_a$, solving the SDE with initial condition $p_a(0)=1$, satisfies $p_a(x)\in [0,2]$ for all $x\in[0,1]$. 
In the following, for each $s\in\mathbb{R}\cup\{\infty\}$ we denote by $\mathbb{P}_s$ the probability distribution of the diffusion process $p(x)$ with initial value $p(0)=s$.
To estimate this probability we use the Cameron-Martin-Girsanov transform
$$\begin{aligned}
&\mathbb{P}_1(p_\lambda(x)\in[0,2]\text{ for all } x\in[0,1])\\
&=\mathbb{E}_1\left[\exp\left(-\frac{1}{\sigma}\int_0^1 (-a-b_x^2)db_x-\frac{1}{2\sigma^2}\int_0^1 (-a-b_x^2)^2dx
\right);b_x\in[0,2]\forall x\in[0,1].
\right]
\end{aligned}
$$ for a standard Brownian motion $b$. On the conditioned event we have
$$\frac{1}{2\sigma^2}\int_0^1 (-a-b_x^2)^2dx=\frac{a^2}{2\sigma^2}+O(a),$$
and by Itô's formula, on the conditioned event,
$$
\int_0^1 (-a-b_x^2)db_x=a(b_0-b_1)-\frac{1}{3}(b_1^3-b_0^3)+\sigma^2\int_0^1 b_xdx=O(a).
$$
Therefore we conclude that
$$\mathbb{P}(\operatorname{RSO}_\sigma<-a)\geq e^{-\frac{a^2}{2\sigma^2}(1+o(1)}.$$

\textbf{Right tail, upper bound} We follow the ideas in \cite{ramirez2011beta}, Section 4 yet the computation in our setting is slightly simpler. Now we fix a value $a>0$, and for any $r\in\mathbb{R}$ denote by $\mathfrak{m}_r$ the passage time to the specified level $r$ of the diffusion process
$$dp(x)=\sigma db_x+(a-p^2(x))dx,\quad p(0)=s,\quad x\in[0,1].$$

Then we have for any $a>>1$ and $c>0$,
\begin{equation}
    \mathbb{P}(\operatorname{SAO}_\sigma>a)=\mathbb{P}_\infty (\mathfrak{m}_{-\infty}\leq 1)\leq\mathbb{P}_{\sqrt{a}-c}(\mathfrak{m}_{-\sqrt{a}}\leq 1).
\end{equation}
For fixed $a>0$ denote by $\mathfrak{m}_\pm=\mathfrak{m}_{\pm \sqrt{a}}$, and consider the event $\mathcal{A}=\left\{\mathfrak{m}_+>\mathfrak{m}_{-}\right\}$, then for any $c>0$,
$$\begin{aligned}
\mathbb{P}_{\sqrt{a}-c}(\mathfrak{m}_{-}\leq 1)&=\mathbb{P}_{\sqrt{a}-c}(\mathfrak{m}_{-}\leq 1,\mathcal{A})+\mathbb{P}_{\sqrt{a}-c}(\mathfrak{m}_{-}\leq 1,\mathcal{A}^c)\\
&\leq \mathbb{P}_{\sqrt{a}-c}(\mathfrak{m}_{-}\leq 1,\mathcal{A})+\mathbb{P}_{\sqrt{a}}(\mathfrak{m}_{-}\leq 1)\\&\leq \mathbb{P}_{\sqrt{a}-c}(\mathfrak{m}_{-}\leq 1,\mathcal{A})+\mathbb{P}_{\sqrt{a}}(\mathfrak{m}_{\sqrt{a}-c}\leq 1)\mathbb{P}_{\sqrt{a}-c}(\mathfrak{m}_{-}\leq 1),
\end{aligned}$$ with inequalities following from the fact that the possibility of hitting any level below the initial value decreases as the initial value increases, and increases when the diffusion process $p_a$ is running on a longer time interval. We claim that

\begin{Claim}\label{claim7.1}
    We can find a sufficiently large $c>0$ so that $\mathbb{P}_{\sqrt{a}}(\mathfrak{m}_{\sqrt{a}-c}>1)$ is bounded uniformly away from zero (that is independent of $a>>c$).
\end{Claim}
Once this claim is justified, we can deduce that there is a numerical constant $c'$ such that 
\begin{equation}
    \mathbb{P}_{\sqrt{a}-c}(\mathfrak{m}_{-\sqrt{a}}\leq 1)\leq c'\mathbb{P}_{\sqrt{a}-c}\left(\mathfrak{m}_{-\sqrt{a}}\leq 1,\mathfrak{m}_{\sqrt{a}}>\mathfrak{m}_{-\sqrt{a}}\right).
\end{equation}
Now we can complete the proof via Girsanov transform:
\begin{equation}
    \mathbb{P}_{\sqrt{a}-c}(\mathfrak{m}_{-}\leq 1,\mathcal{A})=\mathbb{E}_{\sqrt{a}-c}[R(q),\mathfrak{m}_{+}>\mathfrak{m}_{-},\mathfrak{m}_{-}\leq 1],
\end{equation}
for $q$ the diffusion with sign-reversed drift
$$dq(x)=\sigma db(x)+(q^2(x)-a)dx,
$$
and the Girsanov change of measure factor is \footnote{The Girsanov density has this form only if $q$ has the sign-reversed drift, and an intuitive explanation showing why choosing $q$ with this sign reversed drift leads to an almost optimal estimate can be found in \cite{dumaz2013right}, Section 2.2. Namely, $q$ gives a very fast trajectory from $\sqrt{a}$ to $-\sqrt{a}$ but does not go below $-\sqrt{a}$ with high probability.}
$$\log R(q)=\frac{2}{\sigma^2}\int_0^{1\wedge\mathfrak{m}_{-}}(a-q^2(x))dq(x).$$
By Itô's lemma, for any $z>0$ we have
\begin{equation}\begin{aligned}
    \int_0^z (a-q^2(x))dq(x)&=a(q(z)-q(0))-\frac{1}{3}(q^3(z)-q^3(0))+\sigma^2 \int_0^z q(x)dx.\end{aligned}
\end{equation}
Whenever $z\leq \mathfrak{m}_{-}\wedge\mathfrak{m}_+$ we have $|q(x)|\leq\sqrt{a}$ for $x\in[0,z]$ so the last term on the right is bounded by a term of order $O(a)$. For the first two terms on the right hand side, note that when $z=\mathfrak{m}_{-}$, since $q(0)=\sqrt{a}-c$, the first line equals $-(4/3)a^{3/2}+O(a)$. Combining everything, we conclude that 
\begin{equation}
    \mathbb{P}(\operatorname{RSO}_\sigma>a)\leq c'e^{-\frac{8}{3\sigma^2}a^{3/2}(1+o(1))}.
\end{equation}

The last remaining step is to prove Claim \ref{claim7.1}.
Indeed, the probability in question is bounded from below by the probability that its (reflected downward once reaching $\sqrt{a}$) version never reaches $\sqrt{a}-c$.
Under these constraints, i.e. $p(x)\in [\sqrt{a}-c,\sqrt{a}]$ for $x\in[0,1]$, the drift of $p$ is bounded from below by $0$. Thus with positive probability $p(x)$ will never reach $\sqrt{a}-c$ on the time interval $[0,1]$, and this probability is uniform for all $a$ sufficiently large.

This completes the proof of Theorem \ref{theorem1.11}.

\section*{Acknowledgements}
The author thanks Professor James Norris for suggesting him to investigate edge scaling limits of random Schrödinger operators.

\section*{Statements and Declarations}

Data sharing not applicable to this article as no datasets were generated or analysed during the current study.

\appendix

\section{Exponential moments of random variables}

We quote the following lemma from \cite{gorin2018stochastic}, Lemma 4.1 that is very useful for exponential moment bounds under Assumption \ref{assumption1.1}:

\begin{lemma}\label{expoest1}
    Assume that for a given random variable $\zeta$, we can find $C>0$, $0<\gamma<2/3$ so that $\mathbb{E}[|\zeta|^\ell]\leq C^\ell \ell^{\gamma^\ell}$ for all $\ell\in\mathbb{N}$. Then we can find $C'>0$, $2<\gamma'<3$ with \begin{equation}
        \mathbb{E}[e^{v\zeta}]\leq\exp(v\mathbb{E}[\zeta]+C'(v^2+|v|^{\gamma'}),\quad v\in\mathbb{R},
    \end{equation}
    \begin{equation}
        \mathbb{E}[|1+v\zeta|^\ell]\leq \exp(|v|\ell |\mathbb{E}[\zeta]|+C'(v^2\ell^2+|v|^{\gamma'}\ell^{\gamma'})),\quad v\in\mathbb{R},\ell\in\mathbb{N}.
    \end{equation}
\end{lemma}

\section{Brownian bridge approximation of random walk} \label{AppendixA}

In this section we outline technical lemmas that guarantee how a random walk bridge can be well approximated by a Brownian bridge. The proof of Proposition \ref{proposition3.1} is given as follows, which is essentially an adaptation of the proof in \cite{gorin2018stochastic}, Appendix B: (note that Appendix B of \cite{gorin2018stochastic} is contained in its ArXiv version, and in the supplementary material downloadable from the Project Euclid website. We follow the numbering of the ArXiv version here).

We begin with a technical proposition from \cite{lawler2007random}, Theorem 6.4:
\begin{proposition} Given any $C>0$, we shall find $\tilde{C}>0$ such that given any $n\in\mathbb{N}$, there is a probability space that supports a random walk bridge $X^{x,y;n,T_n}$ and a Brownian bridge $\widetilde{B}^{x,y;n}$ that connects $\lfloor nx\rfloor$ to $\lfloor ny\rfloor$ in $T_nn^2$ time, that satisfies 
\begin{equation}
    \mathbb{P}\left(\sup_{0\leq t\leq T_n}|X^{x,y;n,T_n}(t)-\widetilde{B}^{x,y;n}(tn^2)|\geq\tilde{C}\log n\right)\leq \tilde{C}n^{-C}.
\end{equation}
Note that by Brownian scaling, for each $n\in\mathbb{N}$, the process $B^{x,y;n}(t):=n^{-1}(\tilde{B}^{x,y;n}(Tn^2)),t\in[0,T_n]$ is a standard Brownian bridge that connects $n^{-1}\lfloor nx\rfloor$ to $n^{-1}\lfloor ny\rfloor.$  By the above Proposition,
    \begin{equation}\label{4.2s}
        \mathbb{P}\left(\sup_{0\leq t\leq T_n}|n^{-1}(X^{x,y;n,T_n}(t)-B^{x,y;n}(t))|\geq \tilde{C}n^{-1}\log n\right)\leq \tilde{C}n^{-C}.
    \end{equation}
\end{proposition}
Consider a probability space that supports a Brownian bridge $B^{x,y}$ connecting $x$ to $y$ within time $T$, then by Brownian scaling, for each $n\in\mathbb{N}$ the process \begin{equation}
    \begin{aligned}
&\left(B^{x,y}(t\frac{T}{T_n})-(1-\frac{t}{T_n})x-\frac{t}{T_n}y\right)(\frac{T}{T_n})^{1/2}\\&
+(1-\frac{t}{T_n})\left(n^{-1}\lfloor nx\rfloor\right)+\frac{t}{T_n}\left(n^{-1}\lfloor ny\rfloor\right),\quad t\in[0,T_n],
    \end{aligned}
\end{equation} has the same law as the Brownian bridge $B^{x,y;n}$ we described just above. Therefore, we start with a probability space on which $B^{x,y}$ is defined, then we enlarge the probability space by using conditioned distributions of $X^{x,y;n,T_n}$ given $B^{x,y;n}$ to obtain copies of random walk bridges $X^{x,y;n,T_n},n\in\mathbb{N}$.
Then the property \eqref{3.13.3} follows from the coupling \eqref{4.2s}, the Borel-Cantelli lemma, and Lévy's modulus of continuity:
\begin{equation}
    \mathbb{P}\left(\lim\sup_{\epsilon\to 0}\frac{\sup_{0\leq t_1\leq t_2\leq t,t_2-t_1\leq\epsilon}|B^{x,y}(t_2)-B^{x,y}(t_1)|}{\sqrt{2\epsilon \log(1/\epsilon)}}=1\right)=1.
\end{equation}

The property \eqref{3.03.2} shall be verified as follows: we quote \cite{gorin2018stochastic}, Lemma B2:
\begin{lemma}
Given $f_1$, $f_2$ measurable functions on $[0,T]$ possessing local times, then given any $\epsilon>0$,
\begin{equation}\label{wegwgag11}
    \sup_{h\in\mathbb{R}}|L_h(f_1)-L_h(f_2)|\leq\frac{T}{\epsilon^2}\sup_{0\leq t\leq T}|f_1(t)-f_2(t)|+\sum_{i=1}^2 \sup_{h_1,h_2\in\mathbb{R},|h_1-h_2|\leq\epsilon} |L_{h_1}(f_i)-L_{h_2}(f_i)|.
\end{equation}
\end{lemma}
Now for any fixed $n\in\mathbb{N}$, we apply the lemma for $f_1:=B^{x,y}$, $f_2:=X^{x,y;n,T_n}$, and $\epsilon=n^{-2/5}$. We use the following estimate
\begin{equation}\label{sgagaga}
    \sup_{h_1,h_2\in\mathbb{R},|h_1-h_2|\leq n^{-2/5}}|L_{h_1}(B^{x,y})-L_{h_2}(B^{x,y})|\leq\mathcal{C}n^{-1/5}(\log n)^{1/2},N\in\mathbb{N},\end{equation}
almost surely, via the following reasoning: first, the laws of $B^{x,y}(t)-x,t\in[0,T/2]$ and $B^{x,y}(T-t)-y,t\in[0,T/2]$ are mutually absolutely continuous towards that of the standard Brownian motion on $[0,T/2]$. Then one should apply a version of estimate \eqref{sgagaga} for standard Brownian motion, which can be found from \cite{trotter1958property}, estimate 2.1. 

We also need the following estimate:

\begin{lemma}\label{lemmab3}
Given any $\tilde{\epsilon}>0$, we can find random variable $\mathcal{C}_{\tilde{\epsilon}}$ so that almost surely,
\begin{equation}\label{2z2z2z}
    \sup_{h_1,h_2\in\mathbb{R},|h_1-h_2|\leq n^{-2/5}}|L_{h_1}(X^{x,y;n,T_n})-L_{h_2}(X^{x,y;n,T_n})|\leq \mathcal{C}_{\tilde{\epsilon}}n^{-1/5+\tilde{\epsilon}},\quad n\in\mathbb{N}.\end{equation}

\end{lemma}
Having proved this lemma, the estimate \eqref{3.03.2} follows from \eqref{3.13.3}, \eqref{sgagaga}, \eqref{wegwgag11}, and selecting $\tilde{\epsilon}>0$ sufficiently small. The proof of Lemma \ref{lemmab3} is given as follows:

\begin{proof} We follow \cite{gorin2018stochastic}, Lemma B3 and only give a sketch. We work with even $\lfloor nx\rfloor$, $\lfloor ny\rfloor,$ and $T_nn^2$, with the other case being the same. First assume $h_1,h_2$ take the form $n^{-1}(2a)$ given $a\in\mathbb{Z}$, so consider simple symmetric random walk with $T_nn^2$ steps from $\lfloor nx\rfloor$, restricted to even times, and conditioned on ending at $\lfloor ny\rfloor$. By \cite{bass1993strong}, Proposition 3.1, we get the estimate
\begin{equation}
    \begin{aligned}
&\mathbb{P}\left(\sup_{h_1,h_2\in n^{-1}(2\mathbb{Z}),|h_1-h_2|\leq n^{-2/5}}|L_{h_1}(X^{x,y;n,T_n})-L_{h_2}(X^{x,y;n,T_n})|\geq n^{-(1-\tilde{\epsilon})/5\lambda}\right)\\&\leq\mathcal{C}n(e^{-\lambda/C}+n^{-14}). 
    \end{aligned}
\end{equation}The multiplicative factor $n$ in the resulting estimate follows from conditioning on the walk ending at $\lfloor ny\rfloor$. Now take $\lambda=n^{\tilde{\epsilon}/2}$ and use the Borel-Cantelli lemma, we have verified estimate \eqref{2z2z2z} for $h_1,h_2\in n^{-1}(2\mathbb{Z})$.

Then we remove constraints on $h_1,h_2$. In the first step assume $h_1,h_2$ have the form $N^{-1}(2a+e)$ given $a\in\mathbb{Z}$ and $e\in(-1,1)\setminus\{0\}$. Taking $e=\frac{1}{2}$ and by a tail bound on binomial distribution, the conditional probability 
\begin{equation}
    \mathbb{P}\left(|L_{n^{-1}(2a+e)}(X^{x,y;n,T_n})-u|\geq n^{-1/5}u^{1/2}\mid L_{n^{-1}(2a)}(X^{x,y;n,T_n})=u\right)
\end{equation}
vanishes faster than a polynomial in the $n\to\infty$ limit. As we have already verified that $L_{n^{-1}(2a)}(X^{x,y;n,T_n})=O(1)$ as $n\to\infty$ for any $a\in\mathbb{Z}$ almost surely, we now condition on this almost sure event and apply Borel-Cantelli lemma to deduce that \eqref{2z2z2z} holds for $h_1,h_2$ of the form $n^{-1}(2a+e)$.

It remains to consider $h_1,h_2$ having the form $n^{-1}(2a+1)$ given $a\in\mathbb{Z}$. This follows from the combinatorial identity
$$\left|L_{n^{-1}(2a+1)}(X^{x,y;n,T_n})-\frac{L_{n^{-1}(2a+\frac{1}{2})}(X^{x,y;n,T_n})+L_{n^{-1}(2a+\frac{3}{2})}(X^{x,y;n,T_n})}{2}\right|
\leq n^{-1}.
$$
\end{proof}

\section{Exponential moment estimates of local times}\label{appendixC}
In this paper we frequently need exponential moment estimates of averages of local times of the random walk bridge. We collect relevant technical lemmas in this section and give a sketch of proof to each of them. These results are adapted from Proposition 4.2 and 4.3 of \cite{gorin2018stochastic}.

We now compute the tails of random walk bridges.
\begin{proposition}\label{asdfgh}
    Given $T_0>0$ and $\theta\in\mathbb{R}$, then uniformly for $x,y\in[0,1]$,  
    \begin{equation}
\sup_{n\in\mathbb{N}}\sup_{\tilde{T}\in\mathcal{T}(x,y;n,T_0)}\mathbb{E}\left[\exp\left(\theta n^{-2}\sum_{i=0}^{\tilde{T}n^2}n^{-1}X^{x,y;n,\tilde{T}}(in^{-2})\right)\right]<\infty,
    \end{equation}
    where $\mathcal{T}(x,y;n,T_0)$ denotes the set of $\tilde{T}\in[0,T_0)$ such that $\tilde{T}n^2$ is an integer that has the same parity as $\lfloor nx\rfloor-\lfloor ny\rfloor.$
\end{proposition}

\begin{proof}
Consider the random walk bridge
$$\tilde{X}^{x,y;n,\tilde{T}}(t):=n^{-1}X^{x,y;n,\tilde{T}}(t),\quad t\in[0,T]$$ with endpoints $x_n$, $y_n$ such that $|x_n-x|\leq n^{-1}$ and $y_n-y|\leq n^{-1}$. Assume for simplicity $\theta>0$ and $x_n\leq y_n$. The random variable in the expectation is upper bounded by $e^{\theta(T_0+n^{-2})\tilde{M}(n,\tilde{T})}$, where $\tilde{M}(n,\tilde{T})$ is defined by $\tilde{M}(n,\tilde{T}):=\max_{t\in[0,\tilde{T}]}\tilde{X}^{x,y;n,\tilde{T}}(t).$ Define a Markov process $Y$ such that $Y(0)=y_n$, that $Y$ moves if and only if $\tilde{X}^{x,y;n,\tilde{T}}(t)\geq y_n$. In that case $Y$ moves up (or down) if and only if $\tilde{X}^{x,y;n,\tilde{T}}$ moves up or down. Then $Y(t)\geq X(t)$ for $t=0,n^{-2},2n^{-2},\cdots,\tilde{T}$. The increments of $Y$ are a simple symmetric random walk, and the maximum $J_{\tilde{T}n^2}$ of a simple random walk with $\tilde{T}n^2$ steps has, by \cite{chen2010random}, Theorem 6.2.1 and inequality 6.2.3,\begin{equation}
    \mathbb{E}[(J_{\tilde{T}n^2})^m]\leq \sqrt{m!}(C\tilde{T}^{1/2}n)^m,\quad m\in\mathbb{N},\tilde{T}\in\mathcal{T}(x,y;n,\infty),N\in\mathbb{N}
\end{equation} for some $C<\infty$. The exponential moment uniform bound now follows.
\end{proof}

We also need the following Proposition, adapted from \cite{gorin2018stochastic}, Proposition 4.3:
\begin{proposition}\label{propc2}
Given any $T_0>0$, $1\leq p<3$ and $\theta>0$, we have the estimate of uniform integrability
\begin{equation}\label{propc2es}
    \sup_{n\in\mathbb{N}}\sup_{\tilde{T}\in\mathcal{T}(x,y;n,T_0)}\mathbb{E}\left[\exp\left(\theta n^{-1}\sum_{h\in n^{-1}\mathbb{Z}}L_h (X^{x,y;n,\tilde{T}})^p\right)\right]<\infty .
\end{equation}
\end{proposition}

Before giving the proof we introduce some useful terminologies.
For a random walk bridge $X^{x,y;n,\tilde{T}},$ we introduce the following two transforms. We define the quantile transform $Q^{n,\tilde{T}}$ of $X^{x,y;n,\tilde{T}}$ (introduced in \cite{10.1214/EJP.v20-3479}) as follows: first find a unique permutation $\kappa$ of $\{1,2,\cdots,\tilde{T}n^2\}$ such that $$l\to X^{x,y;n,\tilde{T}}(\kappa(l)n^{-2})$$ is increasing in $l$, 
and that if $l_1$, $l_2$ maps to the same value then $\kappa(l_1)<\kappa(l_2)$ whenever $l_1<l_2$. In this way, the quantile transform $Q^{n,\tilde{T}}$ satisfies $Q^{^{N,\tilde{T}}}(0)=0,$ and \begin{equation}
    Q^{n,\tilde{T}}(ln^{-2})=\sum_{l_1=1}^l\left(X^{x,y;n,\tilde{T}}(\kappa(l_1)n^{-2})-X^{x,y;n,\tilde{T}}(\kappa(l_1-1)n^{-2}),\right),\quad l=1,2\cdots,\tilde{T}n^2
.\end{equation}
Define also the Vervaat transform (introduced in \cite{vervaat1979relation}) $V^{n,\tilde{T}}$ obtained from $X^{x,y;n,\tilde{T}}$ via splitting the path $X^{x,y;n,\tilde{T}}$ at the first point it achieves the global minimum, then attach its first part at the end of its second part, then moving the resulting path by a constant so the starting point is zero. Now we give the proof of Proposition \ref{propc2}:

\begin{proof}
We essentially follow the proof of Proposition 4.3 in \cite{gorin2018stochastic}. Given any $h\in n^{-1}\mathbb{Z}$, denote by $u_h^{n,\tilde{T}}$, $d_h^{n,\tilde{T}}$ the numbers of upward and downward steps of $X^{x,y;n,\tilde{T}}$ whose previous step is $n-nh$, and consider 
\begin{equation}
    t_h^{n,\tilde{T}}=\sum_{n^{-1}\mathbb{Z}\ni h_1>h} \left(u_{h_1}^{n,\tilde{T}}+d_{h_1}^{n,\tilde{T}}\right),
\end{equation}
 from the classical inequality $(a+b)^{c}\leq 2^c(a^c+b^c)$ with positive $a,b,c$, we deduce that 
\begin{equation}
\begin{aligned}
    &\sum_{h\in n^{-1}\mathbb{Z}}L_h(X^{x,y;n,\tilde{T}})^p=n^{-p}\sum_{h\in n^{-1}\mathbb{Z}}{(u_h^{n,\tilde{T}}+d_h^{n,\tilde{T}})}^p\\
    &\leq n^{-p}2^{p-1}\sum_{h\in n^{-1}\mathbb{Z}}\left( (u_h^{n,\tilde{T}}+d_h^{n,\tilde{T}})(u_h^{n,\tilde{T}})^{p-1}+(u_h^{n,\tilde{T}}+d_h^{n,\tilde{T}})(d_h^{n,\tilde{T}})^{p-1}
    \right).
\end{aligned}
\end{equation}
    By the combinatorial identity in \cite{10.1214/EJP.v20-3479} (5.3),
    $$\begin{aligned}
Q^{n,\tilde{T}}(t^{n,\tilde{T}}_{h-n^{-1}}n^{-2})=u_h^{n,\tilde{T}}+(n-nh-\lfloor nx\rfloor)_{+}-(n-nh-\lfloor ny\rfloor)_{+},\quad h\in n^{-1}\mathbb{Z},
    \end{aligned}$$
    which leads to the estimate
    \begin{equation}
        u_h^{n,\tilde{T}}\leq Q^{n,\tilde{T}}(t_{h-n^{-1}}^{n,\tilde{T}}n^{-2})+n|x-y|+1,\quad h\in n^{-1}\mathbb{Z}.
    \end{equation}
    Since we have the restriction $|d_{h-n^{-1}}^{n,\tilde{T}}-u_h^{n,\tilde{T}}|\leq 1$ for all $h\in n^{-1}\mathbb{Z}$, we have
    \begin{equation}
        d_h^{n,\tilde{T}}\leq  Q^{n,\tilde{T}}(t_h^{n,\tilde{T}}n^{-2})+n|x-y|+2,\quad h\in n^{-1}\mathbb{Z}.
    \end{equation}Combining this with $u_h^{n,\tilde{T}}+d_h^{n,\tilde{T}}=t^{n,\tilde{T}}_{h-n^{-1}}-t_h^{n,\tilde{T}},h\in n^{-1}\mathbb{Z}$, we deduce that 

$$\begin{aligned}
&\sum_{h\in n^{-1}\mathbb{Z}}L_h(X^{x,y;n,\tilde{T}})^p
\\&\leq n^{-p}2^{p-1}\sum_{h\in n^{-1}\mathbb{Z}}(t^{n,\tilde{T}}_{h-n^{-1}}-t_h^{n,\tilde{T}})\left(Q^{n,\tilde{T}}(t_{h-n^{-1}}^{n,\tilde{T}}n^{-2})+n|x-y|+1\right)^{p-1}\\
&+n^{-p}2^{p-1}\sum_{h\in n^{-1}\mathbb{Z}}(t_{h-n^{-1}}^{n,\tilde{T}}-t_h^{n,\tilde{T}})\left(Q^{n,\tilde{T}}(t_h^{n,\tilde{T}}n^{-2})+n|x-y|+2\right)^{p-1}.
\end{aligned}$$
Denoting by $M(n,\tilde{T})$ the maximal value taken by $n^{-1}Q^{n,\tilde{T}},$ using 
$$\sum_{h\in n^{-1}\mathbb{Z}}\left(t^{n,\tilde{T}}_{h-n^{-1}}-t_h^{n,\tilde{T}}\right)=\tilde{T}n^2,$$ we get 
\begin{equation}
    n^{-1} \sum_{h\in n^{-1}\mathbb{Z}}L_h(X^{x,y;n,\tilde{T}})^p\leq 2^{2p-1}\tilde{T}\left(M(n,\tilde{T})^{p-1}+(|x-y|+2n^{-1})^{p-1}\right). 
\end{equation}
By \cite{10.1214/EJP.v20-3479} Corollary 7.4, the distribution of $M(n,\tilde{T})$ is identical to the distribution of the maximum of the Vervaat transform after normalization, that is, to $n^{-1}V^{n,\tilde{T}}$. By definition, the maximum of $V^{n,\tilde{T}}$ is equal to the width of $X^{x,y;n,\tilde{T}}.$ By the same proof as in Proposition \ref{asdfgh}, we get a uniform in $n$ estimate of the $(p-1)$-st moment of the width of simple symmetric random walks, normalized by $n^{-1}$, with $\tilde{T}n^2$ steps. This completes the proof.
\end{proof}

We note in passing a useful remark that will be used in Section \ref{section3}. 
\begin{remark}
   The estimate \eqref{propc2es} also holds if we replace $\sum_{h\in n^{-1}\mathbb{Z}}$ by $\sum_{h\in c+n^{-1}\mathbb{Z}}$ given some $c\in\mathbb{R}$. This follows from the combinatorial constraint
   $$ L_h(X^{x,y;n,\tilde{T}})\leq L_{n^{-1}\lfloor nh\rfloor}(X^{x,y;n,\tilde{T}})+L_{n^{-1}\lceil nh\rceil }(X^{x,y;n,\tilde{T}}).
   $$
\end{remark}

\printbibliography
\end{document}